\documentclass[11pt]{article}
\usepackage{amsfonts}
\usepackage{amsmath}
\usepackage{amssymb}
\usepackage{amsthm}
\usepackage{amsmath}
\usepackage{verbatim}
\usepackage{bibspacing}
\usepackage[active]{srcltx}
\usepackage{hyperref}
\usepackage{authblk}

\newtheorem{thm}{Theorem}[section]
\newtheorem*{thm*}{Theorem}
\newtheorem{cor}[thm]{Corollary}

\newtheorem{lem}[thm]{Lemma}

\newtheorem{prop}[thm]{Proposition}
\newtheorem*{prop*}{Proposition}

\newtheorem*{conj*}{Conjecture}
\newtheorem{defn}[thm]{Definition}
\newtheorem*{dfn*}{Definition}
\theoremstyle{definition}
\newtheorem{rem}[thm]{\textbf{Remark}}
\newtheorem*{rmk*}{Remark}
\newtheorem*{fact*}{Fact}

\theoremstyle{proof}

\renewcommand{\labelenumi}{(\arabic{enumi})}

\numberwithin{equation}{section}

\newcommand{\floor}[1]{\lfloor #1 \rfloor}

\newcommand{\norm}[1]{\left\Vert#1\right\Vert}
\newcommand{\snorm}[1]{\Vert#1\Vert}
\newcommand{\abs}[1]{\left\vert#1\right\vert}
\newcommand{\set}[1]{\left\{#1\right\}}
\newcommand{\brac}[1]{\left(#1\right)}
\newcommand{\scalar}[1]{\left \langle #1 \right \rangle}

\newcommand{\Real}{\mathbb{R}}

\newcommand{\M}{\mathbb{M}}

\newcommand{\eps}{\epsilon}
\renewcommand{\det}{{\rm det}}
\newcommand{\Cov}{{\rm Cov}}

\newcommand{\vrad}{{\rm vrad}}
\newcommand{\volrad}{\vrad}
\newcommand{\tr}{{\rm tr}}
\newcommand{\Eps}{\mathcal{E}}
\newcommand{\E}{\mathbb{E}}
\renewcommand{\P}{\mathbb{P}}
\renewcommand{\L}{\mathcal{L}}
\renewcommand{\M}{\mathcal{M}}

\begin{document}

\title{Generalized Dual Sudakov Minoration via Dimension Reduction - A Program}

\author{Shahar Mendelson\textsuperscript{1}, Emanuel Milman\textsuperscript{2} and Grigoris Paouris\textsuperscript{3}}

\footnotetext[1]{Department of Mathematics, Technion - Israel Institute of Technology, Haifa 32000, Israel, and Mathematical Sciences Institute, The Australian National University, Canberra ACT  2601, Australia. Email: shahar@tx.technion.ac.il.}
\footnotetext[2]{Department of Mathematics, Technion - Israel Institute of Technology, Haifa 32000, Israel. Email: emilman@tx.technion.ac.il.}  
\footnotetext[3]{Department of Mathematics, Texas A\&M University, College Station, TX 77843, USA. Email: grigorios.paouris@gmail.com.\\
S.M. was supported in part by the Israel Science Foundation. G.P. was supported by the US NSF grant CAREER-1151711. E.M. and G.P. were supported by BSF (grant no. 2010288). The research leading to these results is part of a project that has received funding from the European Research Council (ERC) under the European Union's Horizon 2020 research and innovation programme (grant agreement No 637851).}

\begingroup    \renewcommand{\thefootnote}{}    
\endgroup

\date{\today}
\maketitle

\begin{abstract}
We propose a program for establishing a conjectural extension to the class of (origin-symmetric) log-concave probability measures $\mu$,
of the classical dual Sudakov Minoration on the expectation of the supremum of a Gaussian process:
\begin{equation} \label{eq:abstract}
M(Z_p(\mu) , C \int \norm{x}_K d\mu \cdot K) \leq \exp(C p) \;\;\; \forall p \geq 1 .
\end{equation}
Here $K$ is an origin-symmetric convex body, $Z_p(\mu)$ is the $L_p$-centroid body associated to $\mu$, $M(A,B)$ is the packing-number of $B$ in $A$, and $C > 0$ is a universal constant. 
The Program consists of first establishing a Weak Generalized Dual Sudakov Minoration, involving the dimension $n$ of the ambient space, which is then self-improved to a dimension-free estimate after applying a dimension-reduction step. The latter step may be thought of as a conjectural ``small-ball one-sided" variant of the Johnson--Lindenstrauss dimension-reduction lemma. We establish the Weak Generalized Dual Sudakov Minoration for a variety of log-concave probability measures and convex bodies (for instance, this step is fully resolved assuming a positive answer to the Slicing Problem). The Separation Dimension-Reduction step is fully established for ellipsoids and, up to logarithmic factors in the dimension, for cubes, resulting in a corresponding Generalized (regular) Dual Sudakov Minoration estimate for these bodies and arbitrary log-concave measures, which are shown to be (essentially) best-possible. Along the way, we establish a regular version of (\ref{eq:abstract}) for all $p \geq n$ and provide a new direct proof of Sudakov Minoration via The Program. 
\end{abstract}

\section{Introduction}

Let $\gamma_n$ denote the standard Gaussian measure on $\Real^n$, and let $K \subset \Real^n$ denote a convex body, that is a convex compact set with non-empty interior. For simplicity, we assume that $K$ is origin-symmetric, $K = -K$, and denote by $\norm{\cdot}_K$ the associated norm whose unit-ball is $K$. The dual norm is denoted $\norm{\cdot}^*_K$. Given two compact sets $A,B \subset \Real^n$, we denote by $M(A,B)$ the packing number of $B$ in $A$, i.e. the maximal integer $M$ so that there exist $\set{x_i}_{i=1,\ldots,M} \subset A$ with $x_i + B$ mutually disjoint (``$\set{x_i}$ are $B$-separated").

\medskip

This paper is dedicated to the study of a conjectural generalized version of the classical Sudakov Minoration estimate \cite{SudakovMinoration} and its dual version (due to Pajor--Tomczak-Jaegermann \cite{PajorTomczakLowMStar}, see also \cite[Chapter 3.3]{LedouxTalagrand-Book}):
\begin{thm*}[Sudakov and Dual Sudakov Minoration]
\hfill
\begin{enumerate} 
\item Sudakov Minoration for $\ell^*(K) := \int \norm{x}_K^* d\gamma_n(x)$:
\[
M(K , t B_2^n) \leq \exp( C \ell^*(K)^2 / t^2) \;\;\; \forall t > 0  . 
\]
\item Dual Sudakov Minoration for $\ell(K) := \int \norm{x}_K d\gamma_n(x)$:
\[
M(B_2^n , t K) \leq \exp( C \ell(K)^2 / t^2) \;\;\; \forall t > 0  .
\]
\end{enumerate}
\end{thm*}

The term ``minoration" refers to the resulting lower bounds on $\ell^*(K)$ and $\ell(K)$ as a function of the packing numbers. 
Here and throughout this work, $C$,$C'$,$C''$,$c$, etc... denote positive universal constants, independent of any other parameter (and in particular the dimension $n$),  whose value may change from one occurrence to the next. We use $A \simeq B$ to signify that $c \leq A/B \leq C$. The Euclidean unit-ball is denoted by $B_2^n$. 

\medskip

Sudakov Minoration plays a key-role in the proof of M.~Talagrand's ``Majorizing Measures" theorem \cite{Talagrand-RegularityOfGaussianProcesses,Talagrand-GenericChaining,Talagrand-Book}, which gives two-sided bounds on the expected supremum of a Gaussian process in terms of a certain geometric parameter associated with the indexing set. Subsequently (see \cite{Talagrand-Book} and the references therein), Talagrand extended his characterization to more general processes sampled from i.i.d. Bernoulli random variables and measures of the form $\exp(-\sum_{i=1}^n \abs{x_i}^p) dx$, $p \in [1,\infty]$, and the case of general log-concave \emph{product} measures (with moderate tail decay) was obtained by R.~Lata{\l}a \cite{Latala-SudakovMinorationAndGenericChaining}.  Recall that a probability measure $\mu$ on $\Real^n$ is called log-concave if $\mu = \exp(-V(x)) dx$ with $V : \Real^n \rightarrow \Real \cup \set{+\infty}$ convex. Motivated by an attempt to extend Talagrand's characterization to more general log-concave measures,  Lata{\l}a \cite{Latala-GeneralizedSudakovMinoration} and independently the authors (unpublished) conjectured the following \emph{Generalized Sudakov Minoration} bounds:

\begin{conj*} 
For any origin-symmetric log-concave probability measure $\mu$ on $\Real^n$ and origin-symmetric convex body $K\subset \Real^n$:
\begin{enumerate}
\item Generalized Sudakov Minoration for $I^*_1(\mu,K) := \int \norm{x}_K^* d\mu(x)$:
\[
M(K , C I_1^*(\mu,K) B_p(\mu)) \leq \exp(C p) \;\;\; \forall p \geq 1.
\]
\item Generalized Dual Sudakov Minoration for $I_1(\mu,K) := \int \norm{x}_K d\mu(x)$:
\[
M(Z_p(\mu) , C I_1(\mu,K) K) \leq \exp(C p) \;\;\; \forall p \geq 1 .
\]
\end{enumerate}
\end{conj*}

Here $B_p(\mu)$ denotes the unit-ball of the norm given by: \[
\norm{x}_{B_p(\mu)} := \brac{\int \abs{\scalar{y,x}}^p d\mu(y)}^{1/p} ,
\]
and $Z_p(\mu)$ denotes the polar-body $B_p(\mu)^\circ$, defined by $\norm{\cdot}_{Z_p(\mu)} = \norm{\cdot}^*_{B_p(\mu)}$. Up to normalization, the $Z_p(\mu)$ bodies coincide with the $L_p$ centroid-bodies introduced by E. Lutwak and G. Zhang \cite{LutwakZhang-IntroduceLqCentroidBodies}, and have played a pivotal role in the development of our understanding of the volumetric properties of log-concave measures in the last decade (e.g. \cite{GreekBook}). Note that when $\mu = \gamma_n$, since $Z_p(\gamma_n) \simeq \sqrt{p} B_2^n$, the above conjecture precisely coincides with the classical Sudakov Minoration and its dual version. Assuming a general positive answer to a conjecture of Pietsch on the duality of entropy-numbers \cite[p. 38]{Pietsch-Book} (cf. \cite{AMS-Duality-For-Ball,AMST-Duality-For-K-Convex,EMilman-Duality-of-Entropy}), the primal version (1) and dual version (2) of the above conjecture are in fact equivalent; for instance, the results of \cite{AMS-Duality-For-Ball} imply that the two versions are equivalent when $K$ is an ellipsoid. We refer to \cite{Talagrand-GenericChaining,Latala-SudakovMinorationAndGenericChaining,Latala-GeneralizedSudakovMinoration,LatalaTkocz-SudakovMinorationForRegularProduct} for further partial results confirming the Generalized Sudakov Minoration conjecture in particular cases and additional applications. Let us presently only mention that this conjecture has been confirmed in the following cases:
\begin{itemize}
\item For all log-concave \emph{product} measures \cite{Latala-SudakovMinorationAndGenericChaining}; in fact, the same holds for more general regular product measures \cite{LatalaTkocz-SudakovMinorationForRegularProduct}.
\item For all log-concave measures when the extremal points of $K$ are antipodal pairs with disjoint supports \cite{Latala-GeneralizedSudakovMinoration}.
\item For $\mu = \exp(-\varphi(\norm{x}_p)) dx$, $p \in [1,\infty]$, with $\varphi : [0,\infty) \rightarrow \Real \cup \set{+\infty}$ non-decreasing and convex \cite{Latala-GeneralizedSudakovMinoration}.
\item For all log-concave measures when $p \geq 2 n \log(e + n)$ \cite{Latala-GeneralizedSudakovMinoration}.
\end{itemize}

\subsection{Johnson--Lindenstrauss Lemma}

In this work, we choose to concentrate on the conjectured Generalized \emph{Dual} Sudakov estimate, and propose a novel program for establishing it. The Program is based on a conjectural dimension reduction step, which may be thought of as a ``\emph{one-sided} Johnson--Lindenstrauss lemma". Recall that the classical lemma of W.~B.~Johnson and J.~Lindenstrauss \cite{JohnsonLindenstraussLemma} asserts that if $\set{x_i}_{i=1,\ldots,e^k}$ is a collection of (say distinct) points in Euclidean space $X = (\Real^n,\abs{\cdot})$ and $\eps \in (0,1)$, then there exists a map $T : X \rightarrow Y$, $Y = (\Real^m,\abs{\cdot})$ Euclidean, so that:
\begin{equation} \label{eq:distance-preserve}
1-\eps \leq \frac{\norm{T x_i - T x_j}_{Y}}{\norm{x_i - x_j}_{X}} \leq 1+ \eps \;\;\; \forall i \neq j
\end{equation}
with $m \leq C k / \eps^2$. Moreover, $T$ may be chosen to be linear, and a random (appropriately rescaled) orthogonal projection does the job with very high-probability (see Lemma \ref{lem:JL}). 

\medskip
For The Program, we will require an extension of this classical result to more general normed spaces. Such a question was studied by Johnson and A.~Naor \cite{JohnsonNaor-JL}, who showed that this is essentially impossible - even for a fixed $\eps \in (0,1)$, if the normed space $Y$ is an $m$-dimensional subspace of $X$ with $m \leq C_\eps k$, the distance-preservation property (\ref{eq:distance-preserve}) for a linear map $T$ implies that $X$ must be \emph{essentially} Hilbertian (see \cite{JohnsonNaor-JL} for the precise formulation, and also \cite[Section 4, Remark 7]{JohnsonNaor-JL} for the case that $Y$ is not assumed to be a subspace of $X$). It follows that when $X = (\Real^n,\norm{\cdot}_K)$, we cannot in general hope for a \emph{two-sided} estimate (\ref{eq:distance-preserve}). 

\medskip
However, for our purposes, we will only need to satisfy the \emph{left-hand-side} inequality in (\ref{eq:distance-preserve}): if the points $\set{x_i}$ are well-separated in $X$, so should their images $\set{T x_i}$ be in $Y$ (``Separation Dimension Reduction"). Of course, without some additional requirement, this is always possible, simply by scaling the norm of $Y$ or the map $T$ in the numerator above. The additional requirement which replaces the right-hand-side inequality in (\ref{eq:distance-preserve}) is that the unit-ball of $Y$ be ``massive enough", as measured with respect to $T_* \mu := \mu \circ T^{-1}$, the push-forward of the measure $\mu$ by $T$,  thereby precluding trivial rescaling attempts. In a sense, this is an averaged variant (with respect to the given measure $\mu$) of the pointwise right-hand-side requirement in (\ref{eq:distance-preserve}). This conjectural ``one-sided Johnson--Lindenstrauss" separation dimension-reduction is in our opinion a fascinating question, which we plan to explore more in depth in the future; it constitutes Part 1 (or more precisely, Part 1') of our proposed program. We are now ready to describe it and the remaining parts of The Program in more detail.

\subsection{The Program - Simplified Version} \label{subsec:TheProgram}

Our proposed program consists of three parts; for simplicity, we describe here a simplified version, postponing a description of the full version to Section \ref{sec:full-program}. 
 Part 1 is a conjectural dimension-reduction step, already alluded to above: if $Z_p(\mu)$ is separated by $e^k$ translates of $K$, then there should be a linear map $T : \Real^n \rightarrow \Real^m$ with $m\simeq k$ so that $T Z_p(\mu) = Z_p(T_* \mu)$ is separated by $e^k$ translates of another star-body $L \subset \Real^m$, which we may choose at our discretion from a family of candidates $\L_m$, with the only requirement being that it should be massive enough with respect to $T_* \mu$. Part 2 consists of establishing a weak version of the Generalized Dual Sudakov estimate for the pair $Z_p(T_* \mu)$ and $L$, which is allowed to depend (in an appropriate manner) on the dimension $m$ of the ambient space. Part 3 consists of establishing the Generalized Dual Sudakov estimate for the latter pair when $p$ is larger than $m$. We will show in Theorem \ref{thm:program} below that the weak estimate of Part 2 may be amplified by means of a bootstrap argument employing Part 1, so as to fit the correct estimate of Part 3, thereby concluding that confirmation of all three parts would imply the Generalized Dual Sudakov Minoration Conjecture. We begin with describing the simplified version of The Program in greater detail. 

\smallskip

A compact set $L \subset \Real^n$ having the origin in its interior is called a star-body if $t L \subset L$ for all $t \in [0,1]$. Given an absolutely continuous probability measure $\mu$ on $\Real^n$, denote $m_q(\mu,L) := \sup \set{ s > 0 ; \mu( s L) \leq e^{-q} }$ so that $\mu(m_q(\mu,L) L) = e^{-q}$. 

\smallskip

Fix an origin-symmetric convex body $K \subset \Real^n$, origin-symmetric log-concave measure $\mu$ on $\Real^n$ and $p \geq 1$. It is known  (see Lemma \ref{lem:Guedon}) that in such a case $I_1(\mu,K) \simeq m_1(\mu,K)$, and so up to universal constants we need not distinguish between these two parameters. 
For all $m = 1,\ldots, n$, set $\M_m := \set{ T_* \mu \; ; \; T : \Real^n \rightarrow \Real^m \text{ linear}}$, which is a family of log-concave measures on $\Real^m$ by the Pr\'ekopa--Leindler theorem (e.g. \cite{GardnerSurveyInBAMS}). In addition, let $\L_m$ denote some family of origin-symmetric star-bodies in $\Real^m$,
so that $K \in \L_n$. The simplified version of The Program for establishing the Generalized Dual Sudakov estimate:
\begin{equation} \label{eq:gen-dual-Sudakov}
\mu(K) \geq \frac{1}{e} \;\; \; \Rightarrow \;\;\; M(Z_p(\mu) , R K) \leq \exp(C_{A,B,\varphi} p ) ,
\end{equation}
consists of establishing the following three parts for some constants $R,A,B \geq 1$ and a certain function $\varphi$, described below; here $k$ is a positive real number. 

\begin{enumerate}
\item \textbf{Part 1 (Massive Separation Dimension Reduction)}. \\
If $M(Z_p(\mu) , R K) = e^k$ with  $\mu(K) \geq \frac{1}{e}$ and $2B \leq k \leq n/A$, show that there exists a linear map $T: \Real^n \rightarrow \Real^m$ and $L \in \L_m$, with $m \leq A k$, so that:
\begin{enumerate}
\item  $M(T Z_p(\mu), L) \geq e^k$ (\textbf{``Separation Dimension Reduction"}).
\item $T_* \mu(L) \geq \exp(-q_m)$, $1 \leq q_m \leq k/2$ (``\textbf{$L$ is sufficiently massive}"). 
\end{enumerate}
\item \textbf{Part 2 (Weak Generalized Dual Sudakov)}. \\
For all $m=1,\ldots,n$, $L \in \L_m$ and $\nu \in \M_m$, show that:
\[
1 \leq p \leq m \;\; , \;\; \nu(L) \geq \exp(-q_m) \;\;\; \Rightarrow \;\;\; M(Z_p(\nu) , L) \leq \exp(q_m + m \varphi(p/m)) ,
\]
where $\varphi : [0,1] \rightarrow \Real_+$ is an increasing function with $\varphi(0) = 0$ and $x \mapsto \varphi(x) / x$ non-increasing (and is independent of all other parameters).
\item \textbf{Part 3 (Large $p$)}. \\
For all $m=1,\ldots,n$, $L \in \L_m$ and $\nu \in \M_m$, show that:
\[
p \geq m \;\; , \;\; \nu(L) \geq \exp(-q_m)  \;\;\; \Rightarrow \;\;\; M(Z_p(\nu) , L) \leq \exp(q_m + B p) . 
\]
\end{enumerate}

\begin{rem}
The following \emph{linear} version of Part 1 should be kept in mind:
\begin{enumerate}
\renewcommand\theenumi{(\arabic{enumi}')}
\renewcommand\labelenumi{\theenumi}
\item 
\textbf{Part 1' - Linear Version} \\
If $\set{x_i}_{i=1,\ldots,e^k} \subset \Real^n$ is a collection of $K$-separated points with $\mu(K) \geq \frac{1}{e}$ and $2B \leq k \leq n/A$, show that there exist a linear map $T: \Real^n \rightarrow \Real^m$ and $L \in \L_m$, with $m \leq A k$, so that:
\begin{enumerate}
\item  $\set{T(x_i)}_{i=1,\ldots,e^k} \subset \Real^m$ are $\frac{1}{R} L$-separated (\textbf{``One-sided Johnson--Lindenstrauss"}).
\item $T_* \mu(L) \geq \exp(-q_m)$, $1 \leq q_m \leq k/2$ (``\textbf{$L$ is sufficiently massive}"). 
\end{enumerate}
\end{enumerate}
By applying this linear version of Part 1 to the maximal collection of $K$-separated points $\set{x_i}$ in $\frac{1}{R} Z_p(\mu)$, it is evident that establishing Part 1' is sufficient for establishing Part 1 of The Program. 
However, this is not an equivalent reformulation, and we will also see in Section \ref{sec:part1-cubes} an example where a non-linear combinatorial argument is required for establishing Part 1. 
\end{rem}

\begin{rem} \label{rem:part2}
Note that using $\varphi(t) = t$ in Part 2 with $m=n$ precisely corresponds to establishing the Generalized Dual Sudakov Minoration conjecture. Part 2 provides the added flexibility of using a weaker function $\varphi$. For instance, using $\varphi(t) = t^q$ for some $q \in (0,1)$ corresponds to establishing $M(Z_p(\nu) , L) \leq \exp(q_m + m^{1-q} p^q)$, i.e. a weak dimension-dependent confirmation of the conjecture for $\nu \in \M_m$ and $L \in \L_m$. 
\end{rem}

\begin{thm}[The Program Yields Generalized Dual Sudakov] \label{thm:program}
Establishing (the simplified version of) The Program above yields the Generalized Dual Sudakov Estimate (\ref{eq:gen-dual-Sudakov}). 
\end{thm}
\begin{proof}
Assume that $\mu(K) \geq 1/e$. We will show that:
\begin{equation} \label{eq:goal-again}
e^k := M(Z_p(\mu) , R K) \leq \exp(C_{A,B,\varphi} p)  ~,~ C_{A,B,\varphi} := \max\brac{2 B , \frac{1}{A \varphi^{-1}(1/(2A))}} .
\end{equation}
Since $p \geq 1$, we may assume that $k \geq C_{A,B,\varphi} \geq 2B$, otherwise there is nothing to prove. 
We now claim there exists a linear map $T : \Real^n \rightarrow \Real^m$ and $L \in \L_m$ for some $m \leq \min(n, A k)$, so that $M(T Z_p(\mu), L) \geq e^k$ and $T_* \mu( L) \geq \exp(-q_m)$, $1 \leq q_m \leq k/2$. Indeed, if $k < n/A$ this follows from Part 1, whereas if $k \geq n/A$ this is actually trivial by using $m=n$, $T = Id$, $L=K$ and $q_m=1$. Denoting $\nu = T_* \mu \in \M_m$, note that $T Z_p(\mu) = Z_p(\nu)$. Consequently, if $p \geq m$ then by Part 3: \[
\exp(k) \leq M(Z_p(\nu) , L) \leq \exp(q_m + B p) \leq \exp(k/2 + B p) ,
\]
implying that $k \leq 2 B p$, as required. Alternatively, if $p \leq m$ then by Part 2 and the assumption that $x \mapsto \varphi(x)/x$ is non-increasing:
\[
\exp(k) \leq M(Z_p(\nu) , L) \leq \exp(q_m + m \varphi(p/m)) \leq \exp(k/2 + A k \varphi(p/(Ak))) .
\]
It follows since $\varphi$ is increasing from $0$ that:
\[
\frac{p}{Ak} \geq \varphi^{-1}\brac{ \frac{1}{2 A} } > 0 ,
\]
implying that $k \leq C_{A,B,\varphi} p$, and concluding the proof. 
\end{proof}

\subsection{Results}

Besides introducing The Program, our main results in this work are as follows:
\begin{enumerate}
\item As a warm-up, we demonstrate in Section \ref{sec:Sudakov} the usefulness of The Program by running an analogous version which yields a new proof of the classical Sudakov Minoration (in fact, an improved version, known to experts). To take care of the Separation Dimension-Reduction step (Part 1), we simply employ the usual Johnson--Lindenstrauss Lemma, while for Parts 2 and 3 we invoke an elementary weak volumetric estimate based on Urysohn's inequality. 
\item In Section \ref{sec:part3}, we establish Part 3 of The Program in full generality, for all (origin-symmetric) log-concave measures $\nu$ and star-bodies $L$ in $\Real^m$. In fact, we obtain the following \emph{regular version} thereof:
\begin{equation} \label{eq:intro-part3}
p \geq m \;\;\; \Rightarrow \;\; \; M(Z_p(\nu) , C t m_q(\nu,L) L) \leq \exp(1 + q+ \frac{p}{t}) \;\;\; \forall t ,q > 0 .
\end{equation}
\item 
In Section \ref{sec:full-program}, we formulate the full version of The Program, which extends the simplified one presented above in two aspects. First, in Part 1, we allow the packing number after dimension reduction to \emph{drop} by a $D$-th root, where $D\geq 1$ is an additional parameter we introduce; this additional flexibility will be crucial for applying The Program to the case of the cube $K = B_\infty^n$, analyzed in Section \ref{sec:part1-cubes}. Second, we also introduce a \emph{regularity} parameter $t > 0$, whose role is to scale the bodies $K$ and $L$. We prove an analogue of Theorem \ref{thm:program}, stating that establishing Parts 1 and 2 of The (full) Program, together with the regular version of Part 3  established in (\ref{eq:intro-part3}), yields a Generalized Regular Dual Sudakov upper bound on $M(Z_p(\mu) , t m_1(\mu,K) K)$ for all $t > 0$.  This is important for obtaining a regular version of the Generalized Dual Sudakov estimate for ellipsoids in Section \ref{sec:part1-ellipsoids}, which is later used for establishing a Weak Generalized Dual Sudakov estimate (Part 2 of The Program) for more general convex bodies in Section \ref{sec:part2}. 
\item In Section \ref{sec:part2-ellipsoids}, we establish Part 2 of The Program for the case that $L$ is an (origin-symmetric) ellipsoid, by invoking a weak volumetric estimate involving all intrinsic volumes of $Z_p(\nu)$. \item In Section \ref{sec:part1-ellipsoids}, we establish the remaining Part 1 of The Program for the case that $K$ is an ellipsoid, by decoupling the separation dimension-reduction and massiveness requirements using a general probabilistic argument, and applying a small-ball one-sided variant of the Johnson--Lindenstrauss Lemma. Running The Program, we obtain the following estimate: 
\begin{equation} \label{eq:intro-ellipsoids}
M(Z_p(\mu) , t m_1(\mu, \Eps) \Eps) \leq \exp \brac{C \brac{ \frac{p}{t^2} + \frac{p}{t} } } \;\;\; \forall t > 0 ,
\end{equation}
for any (origin-symmetric) ellipsoid $\Eps \subset \Real^n$. We verify in Section \ref{sec:conclude} that for general log-concave measures $\mu$ and ellipsoids $\Eps$, this estimate is best-possible (up to numeric constants) for all $p \in [1,n]$ and $t \geq \sqrt{p/n}$. 
When $\mu$ has identity covariance matrix (``$\mu$ is isotropic") and $\Eps = B_2^n$, we have $m_1(\mu,\Eps) \simeq \sqrt{n}$, and so the estimate (\ref{eq:intro-ellipsoids}) precisely coincides with the one obtained in \cite{GPV-ImprovedPsi2} for \emph{isotropic} log-concave measures and Euclidean balls. An alternative proof of this particular case was obtained in \cite{GPV-DistributionOfPsi2} using a very similar approach to the one we employ in this work, namely self-improving a weak Sudakov Minoration estimate via dimension-reduction. In the isotropic case, further improved packing estimates (for an appropriate range of $p,t$) were obtained in \cite[Subsection 3.3]{EMilman-IsotropicMeanWidth}. However, we do not know how to extend the approaches of \cite{GPV-ImprovedPsi2,EMilman-IsotropicMeanWidth} to the general non-isotropic case so that our sharp estimate (\ref{eq:intro-ellipsoids}) is recovered (see Subsection \ref{subsec:conclude-ell} for more details). 
\item In Section \ref{sec:pure}, we introduce the class of $h$-pure log-concave probability measures $\mu$, which includes several important sub-families, such as unconditional, sub-Gaussian and super-Gaussian log-concave measures. In particular, a log-concave measure is called $1$-pure if all of its lower-dimensional marginals have uniformly bounded isotropic constant (see Section \ref{sec:pure}). A regular packing estimate for $M(B_2^n, t Z_n(\mu))$ when $\mu$ is an isotropic $1$-pure log-concave probability measure was obtained by Giannopoulos--Milman in \cite{GiannopoulosEMilman-IsotropicM}, and we extend it here to general $h$-pure measures, as it plays an important role in the subsequent section. 
\item In Section \ref{sec:part2}, we use the previous results to establish Part 2 of The Program in a variety of scenarios, such as when the log-concave measure $\nu$ is assumed $h$-pure, or when $Z_m(\nu)$ or $L \in \L_m$ are assumed to have regular small-diameter, such as for type-2 convex bodies, sub-Gaussian convex bodies or unconditional convex bodies with small-diameter, and in particular for $\ell_q^m$ unit-balls, $q \in [2,\infty]$. In view of Remark \ref{rem:part2}, this confirms the Generalized Dual Sudakov Minoration conjecture for such $\nu$ and $L$ up to non-trivial, but unfortunately dimension-dependent, constants. In particular, assuming a positive answer to the Slicing Problem (see Section \ref{sec:pure}), we confirm the conjecture up to non-trivial dimension-dependent constants, a highly non-trivial challenge which constitutes one of the main results of this work. 
\item In Section \ref{sec:part1-cubes} we establish Part 1 of (the full version of) The Program for $K = B_\infty^n$, the $n$-dimensional cube, with additional logarithmic terms in the dimension. Running The Program, this yields for all $p \geq 1$ and $t > 0$:
\begin{equation} \label{eq:intro-cubes}
M(Z_p(\mu) , t C \log \log (e+n) m_1(\mu, B_\infty^n) B_\infty^n) \leq \exp \brac{C \log(e+n) \brac{ \frac{p}{t^2} + \frac{p}{t} } } .
\end{equation}
In Section \ref{sec:conclude}, we verify that for general log-concave measures $\mu$ and up to the above logarithmic terms, this estimate is best-possible (up to numeric constants) for all $p \in [1,n]$ and $t \geq \min(1,\sqrt{p/n^\alpha})$, for any fixed $\alpha \in (0,1)$. 
Removing these logarithmic terms would establish the Generalized Dual Sudakov conjecture in full generality, since any origin-symmetric convex-body $K \subset \Real^n$ may be approximated by an $n$-dimensional section of $B_\infty^N$ as $N \rightarrow \infty$ (in fact, using $N = e^n$ would be enough).
A similar argument verifies that (\ref{eq:intro-cubes}) also holds with $B_\infty^n$ replaced by any origin-symmetric polytope with $n^\beta$ facets, for any fixed $\beta \geq 1$ (see Corollary \ref{cor:RegularSudakovPolytopes}). 
So from an optimistic perspective, we are only $\log N$ far from establishing the conjecture, where $N$ is the dimension of the cube into which $K$ (isomorphically) embeds. 
 \end{enumerate}

In Section \ref{sec:conclude} we present some further concluding remarks. 

\medskip
\noindent
\textbf{Acknowledgements.} We thank the anonymous referee for the meticulous reading of the manuscript and for providing many useful comments.

\section{Notation} \label{sec:prelim}

We work in Euclidean space $(\Real^n,\abs{\cdot})$, where $\abs{\cdot}$ denotes the standard Euclidean norm. The Euclidean unit-ball is denoted by $B_2^n$ and the Euclidean unit-sphere by $S^{n-1}$. We also use $\abs{A}$ to denote the volume (or Lebesgue measure) of a Borel set $A \subset \Real^n$ in its $m$-dimensional affine hull (there will be no ambiguity with this standard double role of $\abs{\cdot}$); the volume-radius of $A$ is then defined as $\vrad(A) := (\abs{A} / \abs{B_2^m})^{1/m}$. It is well-known that $\abs{B_2^m}^{1/m} \simeq 1 / \sqrt{m}$. 

The Grassmannian of all $m$-dimensional linear subspaces of $\Real^n$ is denoted by $G_{n,m}$, $m=1,\ldots,n$. All homogeneous spaces $G$ of the group of rotations $SO(n)$ are equipped with their Haar probability measures $\sigma_G$, and in particular $\sigma = \sigma_{S^{n-1}}$ denotes the corresponding Haar probability measure on $S^{n-1}$. Given $F \in G_{n,m}$, we denote by $P_F$ the orthogonal projection onto $F$, and set $B_2(F) := B_2^n \cap F$ and $S(F) := S^{n-1} \cap F$. Given a Borel measure $\mu$ on $\Real^n$, its marginal $\pi_F \mu$ is defined as the push-forward $(P_F)_*(\mu) = \mu \circ P_F^{-1}$. A consequence of the Pr\'ekopa--Leindler celebrated extension of the Brunn--Minkowski inequality (e.g. \cite{GardnerSurveyInBAMS}), is that the marginal $\pi_F \mu$ of a log-concave measure $\mu$ is itself log-concave on $F$.

The support function of a compact set $L$ is defined as $h_L(\theta) := \max \set{\scalar{x,\theta} ; x \in L}$, $\theta \in S^{n-1}$.
Recall that a star-body $L \subset \Real^n$ is a compact set containing the origin in its interior so that $t L \subset L$ for all $t \in [0,1]$. We denote $\norm{x}_L := \min \set{t > 0 ; x \in t L}$. The radial function $\rho_L(\theta)$ is defined as $1/\norm{\theta}_L$ for $\theta \in S^{n-1}$.  When $K$ is an origin-symmetric convex body, $\norm{\cdot}_K$ is a genuine norm whose unit-ball is precisely $K$, and its support function coincides with the dual-norm $h_K(\theta) = \norm{\theta}_K^*$. The Minkowski sum of two compact sets $A,B \subset \Real^n$ is defined as the compact set $A + B := \set{ a+b \; ; \; a \in A ~,~ b \in B}$, and satisfies $h_{A+B} = h_A + h_B$. We will write $L_1 \simeq L_2$ if $c L_2 \subset L_1 \subset C L_2$ for some universal constants $c,C >0$. 

\subsection{Quantiles of log-concave probability measures}

Given an absolutely continuous probability measure $\mu$ on $\Real^n$, a star-body $L\subset \Real^n$ and $q > 0$, recall that:
\[
m_q(\mu,L) := \sup \set{ s > 0 ; \mu( s L) \leq e^{-q} } ,
\]
so that $\mu(m_q(\mu,L) L) = e^{-q}$. In addition, given $q > -1$, we define:
\[
I_q(\mu,L) := \brac{\int \norm{x}_L^q d\mu(x)}^{1/q} . 
\]
 
\begin{lem} \label{lem:Guedon}
Let $K$ be an origin-symmetric convex body and let $\mu$ denote a log-concave probability measure on $\Real^n$. Then for all $q \geq 1$: 
\[
c e^{-q} m_1(\mu,K) \leq m_q(\mu,K) \leq m_1(\mu,K) \simeq I_1(\mu,K) \leq I_q(\mu,K) \leq C q I_1(\mu,K) .
\]
\end{lem}
\begin{proof}
The first inequality follows by a Kahane--Khintchine-type inequality for negative moments due to  O.~Gu\'edon \cite{Guedon-extension-to-negative-p}, which asserts that under our assumptions:
\[
\mu(\eps \; m_1(\mu,K) K) \leq 2 \ln (\frac{e}{e-1}) \eps \;\;\; \forall \eps \in [0,1] .
\]
The second inequality is trivial. The inequality $m_1(\mu,K) \leq \frac{e}{e-1} I_1(\mu,K)$ follows directly by the Markov-Chebyshev inequality. The reverse inequality $I_1(\mu,K) \leq C m_1(\mu,K)$ follows again by Markov-Chebyshev in conjunction with the negative moment comparison $I_1(\mu,K) \leq C_q I_{q}(\mu,K)$ for all $q \in (-1,0]$ established in \cite{Guedon-extension-to-negative-p}. The inequality $I_1(\mu,K) \leq I_q(\mu,K)$ is immediate by Jensen's inequality. Finally, the Kahane--Khintchine-type inequality $I_q(\mu,K) \leq C q I_1(\mu,K)$ is a known consequence of Borell's lemma \cite{Borell-logconcave}  (e.g. 
\cite{Milman-Pajor-LK}, \cite[Appendix III]{Milman-Schechtman-Book} or \cite[Theorem 2.4.6]{GreekBook}). 
\end{proof}

\subsection{Centroid Bodies}

Recall that the $L_p$ ($p\geq 1$) centroid-bodies $Z_p(\mu)$ associated to a log-concave probability measure $\mu$ on $\Real^n$ are defined by:
\[
h_{Z_p(\mu)}(\theta) = \brac{\int \abs{\scalar{x,\theta}}^p d\mu(x) }^{1/p} \;\; , \;\; \theta \in S^{n-1} . 
\]
 Note that $T(Z_p(\mu)) = Z_p(T_* \mu)$ for any linear mapping $T$, and in particular 
$P_F Z_p(\mu) = Z_p(\pi_F \mu)$ for all $F \in G_{n,m}$. 
It is well-known that:
\begin{equation} \label{eq:Zpq}
1 \leq p \leq q  \;\; \Rightarrow \;\; Z_p(\mu) \subset Z_q(\mu) \subset C \frac{q}{p} Z_p(\mu) ,
\end{equation}
for some constant $C \geq 1$. The first inequality is simply Jensen's inequality, whereas the second one is due to Berwald \cite{BerwaldMomentComparison}, or may be deduced as in Lemma \ref{lem:Guedon} as a consequence of Borell's Lemma \cite{Borell-logconcave}. In fact, it was noted by Lata{\l}a and Wojtaszczyk \cite[Proposition 3.8]{LatalaJacobInfConvolution} that when $\mu$ is origin-symmetric, one may use $C=1$ above (note that the argument in \cite{LatalaJacobInfConvolution} applies to the entire range $1 \leq p \leq q$). 

\begin{lem} \label{lem:IpWish1}
For any probability measure $\mu$, origin-symmetric convex body $K$ and $p \geq 1$:
\[
Z_p(\mu) \subset I_p(\mu,K) K .
\]
\end{lem}
\begin{proof}
For all $\theta \in S^{n-1}$:
\[
h^p_{Z_p(\mu)}(\theta) = \int \abs{\scalar{x,\theta}}^p d\mu(x) \leq \int \norm{x}_K^p d\mu(x) \; h_K^p(\theta)  ~.
\]
\end{proof}

\medskip

We denote by $\Cov(\mu)$ the covariance matrix of $\mu$, defined as $\Cov(\mu) := \int x \otimes x \; d\mu(x) - \int x \; d\mu(x) \otimes \int x \; d\mu(x)$. 
We will say that $\mu$ is isotropic if its barycenter is at the origin and $\Cov(\mu)$ is the identity matrix $Id$. 
It is easy to see that by applying an affine transformation, any absolutely continuous probability measure may be brought to isotropic ``position", which is unique up to orthogonal transformations. We will always assume that $\mu$ has barycenter at the origin, so that $\mu$ is isotropic if and only if $Z_2(\mu) = B_2^n$; more generally, we always have $Z_2(\mu) = \Cov(\mu)^{1/2}(B_2^n)$, so that $\abs{Z_2(\mu)} = \abs{B_2^n} (\det \; \Cov(\mu))^{1/2}$ (where we identified between a matrix and its associated linear operator). 

\subsection{Packing and Covering Numbers} \label{subsec:pack-cover}

Recall that given two compact sets $A,B \subset \Real^n$, the packing number $M(A,B)$ of $B$ in $A$ is defined as the maximal integer $M$ so that there exist $\set{x_i}_{i=1,\ldots,M} \subset A$ with $x_i + B$ mutually disjoint (``$\set{x_i}$ are $B$-separated"); note that a more standard definition in the literature is to assume that $x_i - x_j \notin \tilde{B}$, which coincides with our definition if $\tilde{B} = B - B$, yielding a factor of $2$ in $B$ if the latter is origin-symmetric. The covering number $N(A,B)$ of $A$ by $B$ is defined as the minimal integer $N$ so that there exist $\set{x_i}_{i=1,\ldots,N}$ with $A \subset \bigcup_{i=1}^N (x_i + B)$. The following relation between packing-numbers and covering-numbers is well-known (see e.g. \cite[Chapter 4]{AGA-Book-I}):
\[
N(A,B-B) \leq M(A,B) \leq N(A,-B)  .
\]
When $B$ is an origin-symmetric convex body $K$, it follows that:
\begin{equation} \label{eq:cover-pack}
N(A,2K) \leq M(A,K) \leq N(A,K) ,
\end{equation}
and so up to this immaterial factor of $2$, we need not distinguish between packing and covering numbers. 

Note that Lemma \ref{lem:IpWish1} implies that $N(Z_p(\mu) , I_p(\mu,K) K) = 1$. The Dual Generalized Sudakov Conjecture asserts that for log-concave measures, it is possible to replace $I_p(\mu,K)$ by $C I_1(\mu,K)$, and still cover $Z_p(\mu)$ with $\exp(C p)$ copies of $C I_1(\mu,K) K$. 

\medskip

Clearly $M(A,B)$ and $N(A,B)$ are both invariant under simultaneously applying a non-singular linear transformation to both $A$ and $B$, and under translation of $A$ or $B$. 
The following triangle inequality for covering numbers is obvious for all compact $A,B,D$:
\[
N(A,B) \leq N(A,D) N(D,B) .
\]
Note that the following variant also holds for packing numbers:
\begin{lem}
\[
M(A,B) \leq N(A,D) M(D,B) .
\]
\end{lem}
\begin{proof}
Let $Z$ denote a $B$-separated set in $A$ of cardinality $M$, and assume that $A \subset \bigcup_{i=1}^N (x_i + D)$. Clearly $(Z - x_i) \cap D$ is a $B$-separated subset of $D$, and hence:
\[
M = \# Z \leq \sum_{i=1}^N \# (Z \cap (x_i + D)) \leq N M(D,B) . 
\]
\end{proof}

We will frequently use the following obvious volumetric estimates:
\begin{equation} \label{eq:packing-vol}
\frac{\abs{A}}{\abs{B}} \leq N(A,B) \;\;\; , \;\;\; M(A,B) \leq \frac{\abs{A+B}}{\abs{B}} .
\end{equation}
In particular, when $K \subset \Real^n$ is an origin-symmetric convex body, we have the standard volumetric estimate:
\begin{equation} \label{eq:volumetric}
\brac{\frac{1}{t}}^n \leq N(K , t K) \leq M(K , (t/2) K) \leq \brac{\frac{1+ t/2}{t/2}}^n \leq \brac{1 + \frac{2}{t}}^n  \;\;\; \forall t \in (0,1] . 
\end{equation}
\begin{lem} \label{lem:extend}
Assume that for some compact $A$ and convex body $K$ in $\Real^n$:
\[
 N(A,t K) \leq \exp(n \varphi(t)) \;\;\; \forall t \geq t_0 ,
 \]
 for some function $\varphi : [t_0,\infty) \rightarrow \Real_+$. Then the same estimate holds for all $t > 0$ after defining:
 \[
 \varphi(t) := \varphi(t_0) + \log(1 + (2t_0)/t) \;\; , \;\;  t \in (0,t_0) .
 \]
\end{lem}
\begin{proof}
\[
N(A,t K) \leq N(A, t_0 K) N(t_0 K , t K) \leq \exp(n \varphi(t_0)) (1+ (2t_0)/t)^n  .
\]
\end{proof}

\section{Warm Up - Sudakov Minoration via The Program} \label{sec:Sudakov}

In this section, we demonstrate the usefulness of The Program by applying an analogous program which yields a new proof of the Sudakov Minoration Inequality. 
Given a compact set $K \subset \Real^n$, denote (half) the mean-width of $K$ by:
\[
M_{\Real^n}^*(K) = M^*(K) := \int_{S^{n-1}} h_K(\theta) d\sigma(\theta) .
\]

\begin{thm}[Improved Sudakov Minoration] \label{thm:Sudakov}
For any compact $K \subset \Real^n$, one has:
\[
M(K , C t M^*(K) B_2^n) \leq \exp\brac{ \frac{n}{\max(t ,t^2)} } \;\;\; \forall t > 0 ,
\]
for some universal constant $C \geq 1$. 
\end{thm}
\begin{rem}
Note that $M^*(K)$ is invariant under taking the convex hull of $K$, and so we may as well assume that $K$ is a convex body above. In that case, it is well-known (e.g. \cite[p. 203]{AGA-Book-I}) and easy to check by polar-integration that $I^*_1(\gamma_n,K) \simeq \sqrt{n} M^*(K)$. Translating the classical Sudakov Minoration stated in the Introduction using the present notation, it asserts that the left-hand-side is majorized by 
$\exp(n / t^2)$. The improved version above for $t \in (0,1)$ is known to experts, and follows from an elementary volumetric argument, reproduced below. 
\end{rem}

The simplest text-book proof of Sudakov Minoration we are aware of is obtained by first establishing a dual version using a covering estimate of Talagrand, and then applying a duality argument due to Tomczak-Jaegermann (see \cite[Chapter 3.3]{LedouxTalagrand-Book},\cite[Chapter 4.2]{AGA-Book-I}). The proof we provide below is very different: we work with the primal version directly; we first establish the easy ``weak" covering estimate $\exp(n / t)$ from elementary volumetric considerations (taking care of the analogues of Part 2 and Part 3 of The Program); and finally self-improve this estimate when $t \geq 1$ by employing dimension reduction (Part 1 of The Program) via the usual Johnson--Lindenstrauss lemma.

\begin{lem}[Weak Sudakov Inequality, folklore] \label{lem:weak-Sudakov}
For any convex body $K \subset \Real^n$:
\[
M(K, t M^*(K) B_2^n) \leq \exp \brac{\frac{n}{t}} \;\;\; \forall t > 0 . 
\]
\end{lem}
\begin{proof}
By linearity of the support functions we have $M^*(K + t L) = M^*(K) + t M^*(L)$ for all $t \geq 0$. We invoke Urysohn's inequality (e.g. \cite{GiannopoulosMilmanHandbook}, \cite[Chapter 6]{Schneider-Book}), which states that $\volrad(K) \leq M^*(K)$. 
Coupled with the standard volumetric covering estimate (\ref{eq:packing-vol}), we obtain:
\[
M(K, t M^*(K) B_2^n) \leq \frac{\abs{K + t M^*(K) B_2^n}}{\abs{t M^*(K) B_2^n}} \leq \brac{\frac{M^*(K + t M^*(K) B_2^n)}{t M^*(K)}}^n = \brac{\frac{1+t}{t}}^n \leq \exp \brac{\frac{n}{t}} . 
\]
\end{proof}

\begin{lem}[Johnson--Lindenstrauss Lemma \cite{JohnsonLindenstraussLemma}] \label{lem:JL}
Let $F \in G_{n,m}$ be a random $m$-dimensional subspace of Euclidean space $(\Real^n,\abs{\cdot})$ distributed according to the Haar probability measure on $G_{n,m}$, $m=1,\ldots,n$, and let $P_F$ denote the orthogonal projection onto $F$. Then for all $x \in S^{n-1}$:
\begin{enumerate}
\item
\[
\P\brac{ \abs{\sqrt{n/m}\abs{P_F x} - 1} \geq \eps } \leq C \exp(- c m \eps^2) \;\;\; \forall \eps > 0 .
\] 
\item
Let $\set{x_i}_{i=1,\ldots, M} \subset \Real^n$ be a collection of (say distinct) points. Then:
\[
\P \brac{  1-\eps \leq \frac{\sqrt{n/m} \abs{P_F x_i - P_F x_j}}{\abs{x_i - x_j}} \leq 1+ \eps \;\;\; \forall i \neq j } \geq 1 - {M \choose 2} C \exp(- c m \eps^2) .
\]
\end{enumerate}
\end{lem}
\begin{proof}[Proof Sketch]
The first assertion follows from concentration on the sphere and the fact that for a fixed $F_0 \in G_{n,m}$, $S^{n-1} \ni x \mapsto \abs{P_{F_0} x}$ is a $1$-Lipschitz function. The second part follows immediately from the first part, linearity, and the union-bound. 
\end{proof}

\begin{proof}[Proof of Theorem \ref{thm:Sudakov}]
When $t \in (0,C_0]$, where $C_0 \geq 1$ is a large-enough constant to be determined, the assertion follows from Lemma \ref{lem:weak-Sudakov}. When $t \geq C_0$, we proceed as follows. Set:
\[
e^k := M(K, t M^*(K) B_2^n) ,
\]
Lemma \ref{lem:weak-Sudakov} ensures that $k \leq n/C_0$, and since the packing number is an integer, we may assume that $k \geq \log 2$ (otherwise $k=0$ and there is nothing to prove). Let $\set{x_i}_{i=1,\ldots,e^k}$ denote the maximal collection of points in $K$ which are $t M^*(K) B_2^n$-separated. By the Johnson--Lindenstrauss Lemma \ref{lem:JL}, we may choose $C_0$ large enough so that setting $m := \lceil C_0 k \rceil \in [C_0 \log 2 ,  n]$, an orthogonal projection $P_F$ onto a randomly selected $F \in G_{n,m}$ with respect to its Haar probability measure $\sigma_{G_{n,m}}$, will satisfy with probability at least $1 - {e^k \choose 2} C\exp(-c m (1/2)^2) > 1/2$ that:
\begin{equation} \label{eq:Sudakov-union1}
 \text{$\set{P_F(x_i)}$ are $\frac{1}{2} t M^*(K) \sqrt{\frac{m}{n}} B_2(F)$-separated} .
\end{equation}
In addition, since $h_{P_F K} = h_K|_ F$, note that:
\begin{align*}
M^*(K) = \int_{S^{n-1}} h_K(\theta) d\sigma_{S^{n-1}}(\theta) & = \int_{G_{n,m}} \int_{S(F)} h_K(\theta) d\sigma_{S(F)}(\theta) d\sigma_{G_{n,m}}(F) \\
& =  \int_{G_{n,m}} M^*_F(P_F K) d\sigma_{G_{n,m}}(F) ,
\end{align*}
and so by the Markov--Chebyshev inequality, 
\begin{equation} \label{eq:Sudakov-union2}
M^*_F(P_F K) \leq 2 M^*(K) 
\end{equation}
with probability at least $1/2$ (in fact, it follows by a result of Klartag--Vershynin \cite[Section 3]{Klartag-Vershynin} that this holds with much higher probability, but this is not required here). By the union bound, it follows that there exists a subspace $F \in G_{n,m}$ for which both (\ref{eq:Sudakov-union1}) and (\ref{eq:Sudakov-union2}) hold. Hence, applying Lemma \ref{lem:weak-Sudakov} to $P_F K$ in $F \in G_{n,m}$:
\[
e^k \leq  M\brac{P_F K , \frac{1}{4} t M^*_F(P_F K) \sqrt{\frac{m}{n}} B_2(F)} \leq \exp \brac{4 \frac{m}{t \sqrt{m/n}}} .
\]
Using that $m \leq C_0 k + 1 \leq (C_0 + 1/\log(2)) k$ and solving for $k$, we obtain:
\[
k \leq C' \frac{n}{t^2} ,
\]
and hence:
\[
M(K, t M^*(K) B_2^n) = e^k \leq \exp\brac{C' \frac{n}{t^2}} \;\;\; \forall t \geq C_0 ,
\]
concluding the proof. 

\end{proof}

\section{Part 3 - the case $p \geq n$} \label{sec:part3}

In this section we establish Part 3 of The Program. In fact, we will establish the following regular version thereof, in preparation for introducing the full version of The Program in the next section. 

\begin{thm} \label{thm:part3}
Let $\mu$ denote an origin-symmetric log-concave probability measure on $\Real^n$, and let $L \subset \Real^n$ denote a star-body. Then for any $p \geq n$, we have:
\[
M(Z_p(\mu) , C t m_q(\mu,L) L) \leq \exp(1 + q + \frac{p}{t}) \;\;\; \forall t ,q > 0 \;
\] 
In particular, if $\mu(L) \geq e^{-p}$ with $p \geq n$ then $M(Z_p(\mu) , C L) \leq \exp(3 p)$. \end{thm}

\medskip

For the proof, we require a bit of preparation, emphasizing that $L$ need not be convex but only be star-shaped, which might be useful for establishing Part 1 of The Program (as indicated by some preliminary attempts we do not describe here). We start with the following variation on Talagrand's proof of the dual Sudakov Minoration (e.g. \cite[Chapter 3.3]{LedouxTalagrand-Book},\cite[Chapter 4.2]{AGA-Book-I}), which was already used by Hartzoulaki in her PhD Thesis \cite{Hartzoulaki-PhD} and subsequently employed by other authors as well (cf. \cite{LitvakMilmanPajor-QuasiConvex,GPV-ImprovedPsi2}). 

\medskip

Recall that $\lambda_K$ denotes the uniform probability measure on the convex body $K$.

\begin{prop} \label{prop:Talagrand}
Let $K$ denote a convex body and let $L$ denote a star-body in $\Real^n$. Then:
\[
M(K , 2 t m_q(\lambda_K , L) L) \leq \exp(1 + q + \frac{n}{t}) \;\;\; \forall q, t > 0 . 
\]
\end{prop}

For the proof, we will utilize the following auxiliary probability measure on $\Real^n$, which may be associated to any star-body $K \subset \Real^n$:
\[
\mu_K := \frac{1}{n! \abs{K}} e^{-\norm{x}_K} dx . 
\]

\begin{lem}
With the same assumptions as in Proposition \ref{prop:Talagrand}:
\[
M(K , t \; m_q(\mu_K,L)  L) \leq \exp(q + \frac{1}{t}) \;\;\; \forall q, t > 0 . 
\]
\end{lem}
\begin{proof}
Let us show the following equivalent formulation:
\begin{equation} \label{eq:Talagrand0}
M(K , r L) \leq \frac{\exp( s / r)}{\mu_K(s L)} \;\;\; \forall r , s > 0 .
\end{equation}
By definition, there exist $M := M(K,r L)$ points $z_1 , \ldots, z_M \in K$ so that the sets $\set{z_i + r L}$ are mutually disjoint. Hence, for all $s > 0$, the sets $\set{\frac{s}{r} z_i + s L}$ are also mutually disjoint. In addition, by convexity of $K$:
\[
\mu_K\brac{\frac{s}{r} z_i + s L} = \frac{1}{n! \abs{K}} \int_{sL} e^{-\norm{\frac{s}{r} z_i + x}_K} dx \geq \frac{1}{n! \abs{K}} e^{-\frac{s}{r} \norm{z_i}_K} \int_{s L} e^{-\norm{x}_K} dx \geq e^{-\frac{s}{r}} \mu_K(s L) . 
\]
Consequently:
\[
1 \geq \sum_{i=1}^M \mu_K\brac{\frac{s}{r} z_i + s L} \geq M e^{-\frac{s}{r}} \mu_K(s L) ,
\]
establishing (\ref{eq:Talagrand0}), as required. 
\end{proof}

\begin{lem}
For all star-bodies $K,L \subset \Real^n$ and $q > 1$, we have:
\[
m_{q-1}(\lambda_K , L) \geq \frac{1}{2n} m_q(\mu_K, L) .
\]
\end{lem}
\begin{proof}
For all $s > 0$:
\[
\mu_K(sL) = \frac{1}{n! \abs{K}} \int_{sL} e^{-\norm{x}_K} dx = \frac{1}{n! \abs{K}} \int_0^\infty \abs{t K \cap sL} e^{-t} dt = \frac{1}{n!} \int_0^\infty t^n e^{-t} \lambda_K(\frac{s}{t} L) dt .
\] 
Applying this to $s := m_q(\mu_K,L)$, we obtain:
\[
e^{-q} = \mu_K(s L) = \frac{1}{n!} \int_0^\infty t^n e^{-t} \lambda_K\brac{\frac{s}{t} L} dt \geq \lambda_K\brac{\frac{s}{2n} L} \frac{1}{n!} \int_0^{2n} t^n e^{-t}  dt \geq \frac{1}{e} \lambda_K\brac{\frac{s}{2n} L} ,
\]
where the very rough estimate $\int_0^{2n} t^n e^{-t}  dt \geq \frac{1}{e} \int_0^\infty t^n e^{-t} dt$ is standard and may be easily verified by direct calculation (or e.g. by Markov's inequality when $n \geq 4$). It follows that $\frac{s}{2n} \leq m_{q-1}(\lambda_K , L)$, as asserted. 
\end{proof}

\begin{proof}[Proof of Proposition \ref{prop:Talagrand}]
Applying the previous two lemmas, the proof is immediate:
\[
M(K , 2 t \; m_q(\lambda_K , L)) \leq M\brac{K , \frac{t}{n} m_{q+1}(\mu_K, L)} \leq \exp(1 + q + \frac{n}{t}) . 
\]
\end{proof}

\medskip

One final ingredient we require for the proof of Theorem \ref{thm:part3} involves the following star-body, introduced by K.~Ball \cite{Ball-kdim-sections} (cf. \cite[Chapter 10]{AGA-Book-I}). Given a probability measure $\mu$ on $\Real^n$ with continuous and exponentially-decaying density $f_\mu$ with $f_\mu(0) > 0$ (``non-degenerate measure"), and $p \geq 1$, denote by $K_p(\mu) \subset \Real^n$ the star-body with radial function: 
\[
\rho_{K_p(\mu)}(\theta) = \brac{ \frac{p}{\max f_\mu} \int_0^\infty r^{p-1} f_\mu(r \theta) dr}^{\frac{1}{p}} ~,~ \theta \in S^{n-1} . 
\]
Note our slightly non-standard normalization involving $\max f_\mu$ instead of $f_\mu(0)$, which seems to be more convenient. 
Integration in polar coordinates immediately verifies that $\abs{K_n(\mu)} = \frac{1}{\max f_\mu}$, and that (cf. \cite{Paouris-IsotropicTail}):
\begin{equation} \label{eq:ZpKp}
Z_p(\lambda_{K_{n+p}(\mu)}) = \brac{\frac{\abs{K_n(\mu)}}{\abs{K_{n+p}(\mu)}}}^{1/p} Z_p(\mu) . 
\end{equation}

\begin{lem}
For any star-body $L \subset \Real^n$:
\[
\mu(L) \leq \lambda_{K_n(\mu)}(L) .
\]
In particular, for all $q > 0$:
\[
m_q(\mu, L) \geq m_q(\lambda_{K_n(\mu)},L) .
\]
\end{lem}
\begin{proof}
Simply note that for all $\theta \in S^{n-1}$:
\begin{align*}
& \int_0^{\rho_L(\theta)} r^{n-1} f_\mu(r \theta) dr \leq \min\brac{(\max f_\mu) \frac{\rho_L^n(\theta)}{n} , \int_0^\infty r^{n-1} f_\mu( r\theta) dr} \\
& = \frac{\max f_\mu}{n} \min(\rho_{L}^n(\theta) , \rho_{K_n(\mu)}^n(\theta)) = \frac{\max f_\mu}{n} \rho^n_{K_n(\mu) \cap L}(\theta) . 
\end{align*}
Integrating the above ray-wise inequality on $S^{n-1}$, we obtain:
\begin{align*}
\mu(L) & = \int_{S^{n-1}} \int_0^{\rho_L(\theta)} r^{n-1} f_\mu(r \theta) dr d\theta \leq \max f_\mu \int_{S^{n-1}} \int_0^{\rho_{K_n(\mu) \cap L}(\theta)} r^{n-1} dr d\theta \\
&  = \max f_\mu \abs{K_n(\mu) \cap L} = \lambda_{K_n(\mu)}(L) ,
\end{align*}
as required. 
\end{proof}

Remarkably, it was observed by K.~Ball \cite{Ball-kdim-sections} that when $\mu$ is an origin-symmetric log-concave probability measure, $K_p(\mu)$ is in fact a convex body for all $p \geq 1$; this was extended in \cite{KlartagPerturbationsWithBoundedLK} to the non-symmetric case (assuming that $f_\mu(0) > 0$). 
Consequently, Proposition \ref{prop:Talagrand} immediately yields the following:
\begin{cor} \label{cor:Talagrand}
For any log-concave probability measure $\mu$ on $\Real^n$ so that $f_\mu(0) > 0$, star-body $L \subset \Real^n$, and $q , t > 0$:
\[
M(K_n(\mu) , 2 t m_q(\mu , L) L ) \leq M(K_n(\mu) , 2 t m_q(\lambda_{K_n(\mu)}, L) L) \leq \exp(1 + q + \frac{n}{t}) .
\]
\end{cor}

It remains to pass from $K_n(\mu)$ to $Z_n(\mu)$ in the packing estimate above. This is standard, but for completeness, and in order to prove an additional estimate we will require later on, we provide a proof. First, it is known \cite{Paouris-Small-Diameter} that:
\begin{equation} \label{eq:ZnLambda}
Z_n(\lambda_K) \simeq \text{conv}(K \cup -K) ,
\end{equation}
for any convex body $K \subset \Real^n$. It is also known (see \cite{BarlowMarshallProschan,Ball-kdim-sections,Milman-Pajor-LK} for the even case and \cite[Lemmas 2.5,2.6]{KlartagPerturbationsWithBoundedLK} or \cite[Lemma 3.2 and (3.12)]{PaourisSmallBall} for the general one, noting our non-standard normalization) that for any log-concave measure $\mu$ on $\Real^n$ whose barycenter is at the origin, we have:
\begin{equation} \label{eq:Kpq}
1 \leq p \leq q  \;\; \Rightarrow \;\;  K_p(\mu) \subset K_q(\mu) \subset \frac{\Gamma(q+1)^{1/q}}{\Gamma(p+1)^{1/p}} e^{n (\frac{1}{p} - \frac{1}{q})}K_p(\mu) . 
\end{equation}
In particular, $K_n(\mu) \simeq K_{2n}(\mu)$. Combining this with (\ref{eq:ZpKp}) and (\ref{eq:ZnLambda}), we obtain for an \emph{origin-symmetric} log-concave measure $\mu$ (for which $f_\mu(0) = \max f_\mu > 0$):
\begin{equation} \label{eq:ZnKn}
Z_n(\mu) = \brac{\frac{\abs{K_{2n}(\mu)}}{\abs{K_{n}(\mu)}}}^{1/n} Z_n(\lambda_{K_{2n}(\mu)}) \simeq Z_n(\lambda_{K_{2n}(\mu)}) \simeq K_{2n}(\mu) \simeq K_n(\mu) .
\end{equation}
Note that the origin-symmetry of $\mu$ was crucially used to ensure $Z_n(\lambda_{K_{2n}(\mu)}) \simeq K_{2n}(\mu)$. It is possible to dispose of this restriction by employing the one-sided variants $Z_n^+(\mu)$ introduced in \cite{GuedonEMilmanInterpolating}, but we do not pursue this here. 

\medskip

Summarizing, we deduce from Corollary \ref{cor:Talagrand} and (\ref{eq:ZnKn}) that for an appropriate constant $C > 0$, we have under the assumptions of Theorem \ref{thm:part3}:
\[
M(Z_n(\mu) , C t m_q(\mu,L) L ) \leq M(K_n(\mu) , 2 t m_q(\mu , L) L ) \leq \exp(1 + q + \frac{n}{t}) ,
\]
concluding the theorem for the case $p=n$. When $p \geq n$, simply use (\ref{eq:Zpq}):
\[
Z_p(\mu) \subset \frac{p}{n} Z_n(\mu) ,
\]
and conclude:
\[
M(Z_p(\mu) , C t m_q(\mu, L) L) \leq M\brac{Z_n(\mu) , C t \frac{n}{p}  m_q(\mu, L) L} \leq \exp( 1 + q + \frac{p}{t}) .
\]
The proof of Theorem \ref{thm:part3} is complete. 

\medskip

Before concluding this section, we also record for future use the following well-known fact (cf. \cite{KlartagMilmanLogConcave,Klartag-Psi2,KlartagCLPpolynomial,PaourisSmallBall}); as we did not find a precise reference, we provide a proof for completeness. 

\begin{lem} \label{lem:ZnHuge}
For any log-concave probability measure $\mu$ on $\Real^n$ with barycenter at the origin we have:
\[
I_1(\mu,Z_n(\mu)) \leq I_n(\mu,Z_n(\mu)) \leq C .
\]
\end{lem}
\begin{proof}
As in the proof of (\ref{eq:ZpKp}), it is immediate to verify by polar-integration that for any non-degenerate measure $\nu$ and star-body $L$ in $\Real^n$:
\[
I_p(\nu,L) = \brac{\frac{\abs{K_{n+p}(\nu)}}{\abs{K_{n}(\nu)}}}^{1/p} I_p(\lambda_{K_{n+p}(\nu)} , L) .
\]
Applying this to $\nu = \mu$, $L = Z_n(\mu)$ and $p=n$, and using that $K_{2n}(\mu) \simeq K_n(\mu)$ by (\ref{eq:Kpq}), we obtain:
\[
I_n(\mu,Z_n(\mu)) \leq C' I_n(\lambda_{K_{2n}(\mu)}, Z_n(\mu)) .  
\]
It remains to use as in (\ref{eq:ZnKn}) that (without any symmetry assumptions):
\[
Z_n(\mu) = \brac{\frac{\abs{K_{2n}(\mu)}}{\abs{K_{n}(\mu)}}}^{1/n} Z_n(\lambda_{K_{2n}(\mu)}) \supset Z_n(\lambda_{K_{2n}(\mu)}) \supset c \; \text{Conv}(K_{2n}(\mu) \cup - K_{2n}(\mu)) \supset c K_{2n}(\mu) .
\]
Consequently:
\[
I_n(\mu,Z_n(\mu)) \leq \frac{C'}{c} I_n(\lambda_{K_{2n}(\mu)}, K_{2n}(\mu)) \leq \frac{C'}{c} ,
\]
concluding the proof. 
\end{proof}

\section{The Program - Full Version} \label{sec:full-program}

We are now ready to state the full version of The Program; the full version extends the simplified one presented in the Introduction by allowing the packing number after dimension reduction to \emph{drop} by a $D$-th root and by introducing an additional scaling parameter $t > 0$. As usual, we fix an origin-symmetric convex body $K \subset \Real^n$, origin-symmetric log-concave measure $\mu$ on $\Real^n$ and $p \geq 1$. 
For all $m = 1,\ldots, n$, set as usual $\M_m := \set{ T_* \mu \; ; \; T : \Real^n \rightarrow \Real^m \text{ linear}}$, which is a family of log-concave measures on $\Real^m$ by the Pr\'ekopa--Leindler theorem. In addition, let $\L_m$ denote some family of origin-symmetric star-bodies in $\Real^m$,
so that $K \in \L_n$. 
The Program for establishing the Generalized Regular Dual Sudakov estimate:
\begin{equation} \label{eq:gen-regular-dual-Sudakov}
\mu(K) \geq \frac{1}{e} \;\; \; \Rightarrow \;\;\; M(Z_p(\mu) , t R K) \leq \exp(C_{A,B,D,\varphi_t,t} p ) , 
\end{equation}
consists of establishing the first 2 parts below for some constants $R,A,B,D \geq 1$ and a certain function $\varphi_t$, described below. 

\begin{enumerate}
\item \textbf{Part 1 (Massive Partial Separation Dimension Reduction)}. \\
If $M(Z_p(\mu) , t R K) = e^k$ with  $\mu(K) \geq \frac{1}{e}$ and $4 B D \leq k \leq n/A$, show that there exists $l \in [k / D , k]$ and a linear map $T: \Real^n \rightarrow \Real^m$ and $L \in \L_m$, with $m \leq A l$, so that:
\begin{enumerate}
\item  $M(T Z_p(\mu), t L) \geq e^l$ (\textbf{``Partial Separation Dimension Reduction"}).
\item $T_* \mu(L) \geq \exp(-q_m)$, $1 \leq q_m \leq l/2$ (``\textbf{$L$ is sufficiently massive}"). 
\end{enumerate}
\item \textbf{Part 2 (Weak Generalized Regular Dual Sudakov)}. \\
For all $m=1,\ldots,n$, $L \in \L_m$ and $\nu \in \M_m$, show that:
\[
1 \leq p \leq m \;\; , \;\; \nu(L) \geq \exp(-q_m) \;\;\; \Rightarrow \;\;\; M(Z_p(\nu) , t L) \leq \exp(B + q_m + m \varphi_t(p/m)) ,
\]
where $\varphi_t : [0,1] \rightarrow \Real_+$ is an increasing function with $\varphi_t(0) = 0$ and $x \mapsto \varphi_t(x) / x$ non-increasing (depending only on $t$ and independent of all other parameters).
\item \textbf{Part 3 (Large $p$)}. \\
For all $m=1,\ldots,n$, $L \in \L_m$ and $\nu \in \M_m$, Theorem \ref{thm:part3} verifies that:
\[
p \geq m \;\; , \;\; \nu(L) \geq \exp(-q_m)  \;\;\; \Rightarrow \;\;\; M(Z_p(\nu) , t L) \leq \exp\brac{1 + q_m + C \frac{p}{t}} ,
\]
for some universal constant $C \geq 1$. 
\end{enumerate}

\begin{rem} 
As in the Introduction, we state the following \emph{linear} version of Part 1:
\begin{enumerate}
\renewcommand\theenumi{(\arabic{enumi}')}
\renewcommand\labelenumi{\theenumi}
\item 
\textbf{Part 1' - Linear Version} \\
If $\set{x_i}_{i=1,\ldots,e^k} \subset \Real^n$ is a collection of $K$-separated points with $\mu(K) \geq \frac{1}{e}$ and $4BD \leq k \leq n/A$,
show that there exist $l \in [k / D , k]$, a linear map $T: \Real^n \rightarrow \Real^m$, and $L \in \L_m$ with $m \leq A l$, so that:
\begin{enumerate}
\item  There exists $I \subset \set{1,\ldots,e^k}$ with $\# I \geq e^l$ so that $\set{T(x_i)}_{i \in I} \subset \Real^m$ are $\frac{1}{R} L$-separated (\textbf{``Partial One-sided Johnson--Lindenstrauss"}).
\item $T_* \mu(L) \geq \exp(-q_m)$, $1 \leq q_m \leq l/2$ (``\textbf{$L$ is sufficiently massive}"). 
\end{enumerate}
\end{enumerate}
By applying this linear version of Part 1 to the maximal collection of $K$-separated points $\set{x_i}$ in $\frac{1}{t R} Z_p(\mu)$, it is evident that establishing Part 1' is sufficient (but not necessary) for establishing Part 1 of The (full) Program. 
\end{rem}

\begin{thm} \label{thm:full-program}
Establishing The (full) Program above yields the Generalized Regular Dual Sudakov Estimate (\ref{eq:gen-regular-dual-Sudakov}) with:
\begin{equation} \label{eq:C-def}
C_{A,B,D,\varphi_t,t} := D \max \brac{\frac{4 \max(C, C' \frac{B}{R})}{t}, \frac{1}{A \varphi_t^{-1}(\frac{1}{4A})}} .
\end{equation}
\end{thm}

\begin{proof}
We assume that $\mu(K) \geq 1/e$. We will first show that:
\begin{equation} \label{eq:goal-again}
e^k := M(Z_p(\mu) , t R K) \leq \exp \brac{ D \max \brac{\frac{4C}{t} , \frac{1}{A \varphi_t^{-1}(\frac{1}{4A})}} p } ,
\end{equation}
under the assumption that $k \geq 4 BD$. 

Under this assumption, there exist $l \in [k / D,k]$, a linear map $T : \Real^n \rightarrow \Real^m$, and $L \in \L_m$ for some $m \leq \min(n, A l)$, so that $M(T Z_p(\mu), t L) \geq e^l$ and $T_* \mu( L) \geq \exp(-q_m)$, $1 \leq q_m \leq l/2$. Indeed, if $k < n/A$ this follows from Part 1, whereas if $k \geq n/A$ this is actually trivial by using $m=n$, $l = k$, $T = Id$, $L=K$ and $q_m=1$. Denoting $\nu = T_* \mu \in \M_m$, note that $T Z_p(\mu) = Z_p(\nu)$. Also note that $l/4 \geq  B \geq 1$. Consequently, if $p \geq m$ then by Part 3: 
\[
\exp(l) \leq M(Z_p(\nu) , t L) \leq \exp(1 + q_m + C \frac{p}{t} ) \leq \exp(1 + l/2 + C \frac{p}{t} ) \leq \exp( l/4 + l/2 + C \frac{p}{t} ) ,
\]
implying that $k \leq D l \leq 4 D C \frac{p}{t}$, as required. Alternatively, if $p \leq m$ then by Part 2 and the assumption that $x \mapsto \varphi_t(x)/x$ is non-increasing:
\[
\exp(l) \leq M(Z_p(\nu) , t L) \leq \exp(B + q_m + m \varphi_t(p/m)) \leq \exp(l/4 + l/2 + A l \varphi_t(p/(A l))) .
\]
It follows since $\varphi_t$ is increasing from $0$ that:
\[
\frac{p}{A l} \geq \varphi^{-1}\brac{ \frac{1}{4 A} } > 0 ,
\]
implying that $k \leq D l \leq  D \frac{1}{A \varphi_t^{-1}(\frac{1}{4A})} p$, and establishing (\ref{eq:goal-again}) under the assumption that $k \geq 4 BD$.

To complete the proof, recall that $Z_p(\mu) \subset I_p(\mu,K) K$ by Lemma \ref{lem:IpWish1}. Since $I_p(\mu,K) \leq C' p \; m_1(\mu,K)\leq C' p$ by Lemma \ref{lem:Guedon}, it follows that the left-hand-side of (\ref{eq:goal-again}) is actually $1$ (equivalently, $k=0$) for $t \geq C' p / R$, in which case there is nothing to prove. On the other hand, in the non-trivial range $t \in (0, C' p / R)$, we have $4 D C' \frac{B}{R} \frac{1}{t} p \geq  4 BD$, which leads to the definition of $C_{A,B,D,\varphi_t,t}$ in (\ref{eq:C-def}) and confirms (\ref{eq:gen-regular-dual-Sudakov})  for all $t > 0$. 
\end{proof}

\section{Part 2 - Weak Generalized Dual Sudakov: Ellipsoids} \label{sec:part2-ellipsoids}

Recall that:
\[
I_q(\mu,K) := \brac{\int_{\Real^n} \norm{x}^q_K d\mu(x)}^{1/q} .
\]
When $K = B_2^n$, we simply denote $I_q(\mu) = I_q(\mu,B_2^n)$. 

\begin{thm} \label{thm:part2-ellipsoids}
Let $\mu$ denote an origin-symmetric log-concave probability measure on $\Real^n$ and let $\Eps \subset \Real^n$ denote an (origin-symmetric) ellipsoid. Then for any $p \in [1,n]$:
\[
M(Z_p(\mu) , t I_1(\mu,\Eps)  \Eps) \leq \exp \brac{ C \frac{p^{2/3} n^{1/3}}{t^{2/3}} + C \frac{\sqrt{p} \sqrt{n}}{t} } \;\;\; \forall t > 0 . 
\]
\end{thm}

Since $M(Z_p(\mu) , t I_1(\mu,\Eps)  \Eps)$ is invariant under simultaneously applying a linear transformation to $\mu$ and $\Eps$, we may and will reduce to the case $\Eps = B_2^n$. 
For the proof, our strategy will be to invoke the standard volumetric estimate on the packing numbers (see Subsection \ref{subsec:pack-cover}):
\[
M(Z_p(\mu) ,  t  B_2^n) \leq \frac{\abs{Z_p(\mu) + t B_2^n} }{ \abs{t  B_2^n} } \;\;\; \forall t > 0 . 
\]
To handle the numerator, we use Steiner's classical formula \cite[Chapter 4]{Schneider-Book}, stating that for any convex body $K \subset \Real^n$:
\[
\abs{K + t B_2^n} = \sum_{k=0}^n {n \choose k} W_k(K) t^{n-k} ,
\]
where $W_k(K)$ denotes the $k$-th quermassintegral (or mixed-volume) of $K$; the latter is often denoted as $W_{n-k}(K)$ in the literature, but we prefer our convention which keeps track of the homogeneity in $K$. Recall that by Kubota's formula (e.g. \cite[Chapter 5]{Schneider-Book}), we have:
\begin{equation} \label{eq:Kubota}
W_k(K) = \frac{\abs{B_2^n}}{\abs{B_2^k}} \int_{G_{n,k}} \abs{P_F K} d\sigma_{G_{n,k}}(F) =  \abs{B_2^n} \int_{G_{n,k}} \volrad(P_F K)^k d\sigma_{G_{n,k}}(F) ,
\end{equation}
where $\sigma_{G_{n,k}}$ denotes the Haar probability measure on $G_{n,k}$ (with the interpretation when $k=0$ that $W_0(K) = \abs{B_2^n}$). 

\medskip
To bound $W_k(Z_p(\nu))$, we will need the following averaged version of \cite[Theorem 2.4]{GiannopoulosPajorPaourisPsi2},\cite[Proposition 3.1]{EMilman-IsotropicMeanWidth}:
\begin{prop} \label{prop:Zp-mixed}
Let $\mu$ denote an origin-symmetric log-concave probability measure on $\Real^n$. Then for all $p \geq 1$ and $k=1,\ldots,n$:
\[
W_k(Z_p(\mu))^{\frac{1}{k}} \leq C \max(\sqrt{p} , p / \sqrt{k}) W_k(Z_2(\mu))^{\frac{1}{k}} . 
\]
\end{prop}
\begin{proof}
It was shown in \cite[Theorem 6.2]{Paouris-IsotropicTail} (see also \cite[Corollary 2.2]{EMilman-IsotropicMeanWidth}) that for any (say origin-symmetric) log-concave probability measure $\eta$ on $\Real^k$, one has:
\[
\vrad(Z_p(\eta)) \leq C \sqrt{p} \; \det \Cov(\eta)^{\frac{1}{2k}}  \;\;\; \forall 1 \leq p \leq k . 
\]
When $p \geq k$, since $Z_p(\eta) \subset \frac{p}{k} Z_k(\eta)$, it follows that:
\[
\vrad(Z_p(\eta))  \leq \frac{p}{k}\vrad(Z_k(\eta)) \leq C \frac{p}{\sqrt{k}} \; \det \Cov(\eta)^{\frac{1}{2k}}  \;\;\; \forall p \geq k . 
\]
Finally, noting that $\det \Cov(\eta)^{\frac{1}{2k}} = \volrad(Z_2(\eta))$ as $Z_2(\eta)$ is an ellipsoid, we obtain:
\[
\vrad(Z_p(\eta)) \leq  C \max(\sqrt{p} , p / \sqrt{k}) \vrad(Z_2(\eta)) \;\;\; \forall p \geq 1. 
\]
Applying the above to $\eta = \pi_F \mu$, using that $Z_q(\pi_F \mu) = P_F Z_q(\mu)$,  integrating over $F \in G_{n,k}$ and applying Kubota's formula (\ref{eq:Kubota}), the assertion readily follows. 
\end{proof}

The quermassintegrals of the ellipsoid $Z_2(\mu)$ are particularly easy to compute using elementary linear algebra, and one may show that $(W_k(Z_2(\mu))/\abs{B_2^n})^{2/k}$ is closely related to the $k$-th root of the $k$-th symmetric elementary polynomial in the eigenvalues of $\Cov(\mu)$ (appropriately normalized). However, we will only require:
\begin{lem} \label{lem:Z2-mixed}
For all log-concave probability measures $\mu$ on $\Real^n$:
\[
\det \; \Cov(\mu)^{\frac{1}{2n}} \leq \brac{\frac{W_k(Z_2(\mu))}{\abs{B_2^n}}}^{\frac{1}{k}} \leq \brac{\frac{1}{n} \tr \; \Cov(\mu)}^{\frac{1}{2}} \;\;\; \forall k=1,\ldots, n . 
\]
\end{lem}
\begin{proof}
By the Alexandrov inequalities for the quermassintegrals \cite[Chapter 6]{Schneider-Book}, we have:
\[
\brac{\frac{W_n(Z_2(\mu))}{\abs{B_2^n}}}^{\frac{1}{n}}  \leq \brac{\frac{W_k(Z_2(\mu))}{\abs{B_2^n}}}^{\frac{1}{k}} \leq \frac{W_1(Z_2(\mu))}{\abs{B_2^n}} ,
\]
so it is enough to calculate the expressions on either side. Indeed:
\begin{align*}
\frac{W_1(Z_2(\mu))}{\abs{B_2^n}} & = \int_{S^{n-1}} h_{Z_2(\mu)}(\theta) d\sigma(\theta) \leq \brac{\int_{S^{n-1}} h^2_{Z_2(\mu)}(\theta) d\sigma(\theta)}^{\frac{1}{2}} \\
&= \brac{\int_{S^{n-1}} \scalar{\Cov(\mu) \theta,\theta} d\sigma(\theta)}^{\frac{1}{2}} = \brac{\frac{1}{n} \tr  \; \Cov(\mu)}^{\frac{1}{2}} ,
\end{align*}
while:
\[
\brac{\frac{W_n(Z_2(\mu))}{\abs{B_2^n}}}^{\frac{1}{n}} = \brac{\frac{\abs{Z_2(\mu)}}{\abs{B_2^n}}}^{\frac{1}{n}} = \det \; \Cov(\mu)^{\frac{1}{2n}}  .
\]
\end{proof}

Finally, it is useful to note that by Lemma \ref{lem:Guedon}:
\[
\brac{ \tr \; \Cov(\mu)}^{\frac{1}{2}} = \brac{\int_{\Real^n} \abs{x}^2 d\mu(x)}^{\frac{1}{2}} = I_2(\mu) \simeq  I_1(\mu) . 
\]
We are now ready to prove Theorem \ref{thm:part2-ellipsoids}.

\begin{proof}[Proof of Theorem \ref{thm:part2-ellipsoids}]
\[
M(Z_p(\mu) , t I_1(\mu)  B_2^n) \leq \frac{\abs{Z_p(\mu) + t I_1(\mu) B_2^n} }{ \abs{t  I_1(\mu) B_2^n} } = \sum_{k=0}^n {n \choose k} \frac{W_k(Z_p(\mu))}{(t I_1(\mu))^k\abs{B_2^n}} .
\]
Employing Proposition \ref{prop:Zp-mixed} and Lemma \ref{lem:Z2-mixed}, we know that:
\[
\brac{\frac{W_k(Z_p(\mu))}{\abs{B_2^n}}}^{1/k} \leq C \max(\sqrt{p} , p / \sqrt{k}) \frac{I_1(\mu)}{\sqrt{n}} \;\;\; \forall k=1,\ldots,n . 
\]
Using the standard estimate ${n \choose k} \leq \brac{\frac{e n}{k}}^k$, we obtain:
\begin{align*}
M(Z_p(\mu) , t I_1(\mu) B_2^n) & \leq 1 + \sum_{k=1}^n \brac{C \frac{e \sqrt{n}}{t k} \max(\sqrt{p},p / \sqrt{k}) }^k \\
& = 1 + \sum_{k=1}^{\floor{p}} \brac{C \frac{e p \sqrt{n}}{t k^{3/2}} }^k +  \sum_{k=\floor{p}+1}^{n} \brac{C \frac{e \sqrt{p} \sqrt{n}}{t k} }^k \\
& =  1 + \sum_{k=1}^{\floor{p}} \brac{C' \frac{p^{2/3} n^{1/3}}{t^{2/3} k} }^{\frac{3}{2} k} +  \sum_{k=\floor{p}+1}^{n} \brac{C \frac{e \sqrt{p} \sqrt{n}}{t k} }^k \\
& \leq \sum_{m=0}^{\infty} \frac{1}{m!} \brac{C'' \frac{p^{2/3} n^{1/3}}{t^{2/3}} }^{m} + \sum_{k=0}^{\infty} \frac{1}{k!} \brac{C'' \frac{\sqrt{p} \sqrt{n}}{t} }^k \\
& = \exp \brac{ C'' \frac{p^{2/3} n^{1/3}}{t^{2/3}} + C'' \frac{\sqrt{p} \sqrt{n}}{t} } . 
\end{align*}
The assertion is thus established for $\Eps = B_2^n$, and hence for arbitrary ellipsoids, as explained above. The proof is complete.
\end{proof}

\section{Part 1' - Separation Dimension Reduction: Ellipsoids} \label{sec:part1-ellipsoids}

\subsection{The Probabilistic Approach}

The following proposition decouples the question of separation dimension reduction and the massiveness requirement of Part 1' of The Program. Not surprisingly, this is achieved by introducing some randomness. 

\begin{prop} \label{prop:part1-prob}
Let $K \subset \Real^n$ denote a star-body, and assume that $\set{x_i}_{i=1,\ldots, M} \subset \Real^n$ is a collection of $K$-separated points. Let $T : \Real^n \rightarrow \Real^m$ denote a random linear map and $L,S \subset \Real^m$ denote two random star-bodies defined on a common probability space, so that $L = L_T$ and $S = S_T$ are measurable functions of $T$ (when equipping the family of star-bodies with the Hausdorff metric). Assume that:
\begin{enumerate}
\item If $x \notin K$ then $\P( T x \in L_T) \leq p_{\text{out}}$. 
\item If $x \in K$ then $\P(T x \in S_T) \geq p_{\text{in}}$.
\end{enumerate}
Then for any Borel probability measure $\mu$ on $\Real^n$, if:
\begin{equation} \label{eq:prob-condition}
M^2 p_{\text{out}} \leq \mu(K) p_{\text{in}} ,
\end{equation}
then there exist a linear map $T^0 : \Real^n \rightarrow \Real^m$ and star-bodies $L^0,S^0 \subset \Real^m$ so that $\set{T^0(x_i)}_{i=1,\ldots,M}$ are $L^0$-separated and $T^0_*(\mu)(S^0) \geq \frac{1}{2} \mu(K) p_{\text{in}}$. 
\end{prop}

\begin{proof}
We may assume that $p_{\text{in}}, p_{\text{out}},\mu(K) > 0$. 
By linearity and the union-bound, the random set $\set{T(x_i)}_{i=1,\ldots,M}$ is clearly $L_T$-separated with probability at least:
\[
1 - {M \choose 2} p_{\text{out}} > 1 - \frac{M^2}{2} p_{\text{out}} ,
\]
so it remains to verify the second requirement. Denoting $G_T := T_*(\mu)(S_T)$, note that:
\[
\E(G_T) = \E \brac{ \int_{\Real^n} 1_{\set{T x \in S_T}} d\mu(x) } = \int_{\Real^n} \P(T x \in S_T) d\mu(x) \geq \int_{K} \P(T x \in S_T) d\mu(x) \geq \mu(K) p_{\text{in}} =: q .
\]
Since $0 \leq G_T \leq 1$ and $\E(G_T) \geq q$, it follows that $\P( G_T \geq q/2) \geq q/2$. Consequently, the assumption (\ref{eq:prob-condition}) guarantees that the event that $\set{T(x_i)}_{i=1,\ldots,M}$ are $L_T$-separated and the one that $G_T \geq q/2$ have non-empty intersection, yielding the claim. 
\end{proof}

\subsection{Part 1' For Ellipsoids}

In view of Proposition \ref{prop:part1-prob}, Part 1' of The Program will follow from the following one-sided variant of the Johnson--Lindenstrauss lemma, which pertains to small-ball probability, see e.g. \cite[Fact 3.2]{MilmanSzarek-GeometricLemma} or \cite[Lemma 8.1.15]{GreekBook}:
 
\begin{lem} \label{lem:small-ball}
Let $T : \Real^n \rightarrow \Real^m$ denote a random orthogonal projection, that is $T = P \circ U$ where $U$ is uniformly distributed on $SO(n)$ and $P$ is the canonical projection on the first $m$ coordinates, $m=1,\ldots,n$. Then for all $x \in S^{n-1}$:
\[
\P\brac{ \sqrt{n/m}\abs{T x}\leq \eps } \leq (C' \eps)^m \;\;\; \forall \eps \in [0,1] .
\] 
\end{lem}

\begin{cor}[Part 1' for Euclidean Ball] \label{cor:part1-ellipsoids}
Let $\set{x_i}_{i=1,\ldots, e^k} \subset \Real^n$ be a collection of $B_2^n$-separated points, $k \in [1,n]$, and let $\mu$ denote a Borel probability measure with $\mu(B_2^n) \geq e^{-q} \geq e^{-k}$. Set $m = \lceil k \rceil$, and denote $L := \sqrt{2} \sqrt{m/n} B_2^m$. Then there exists an orthogonal projection $T : \Real^n \rightarrow \Real^m$ (as above) so that:
\begin{enumerate}
\item $\set{T(x_i)}_{i=1,\ldots, e^k} \subset \Real^m$ are $L/C$-separated.
\item $T_*(\mu)(L) \geq \frac{1}{4} e^{-q}$. 
\end{enumerate}
\end{cor}
\begin{proof}
By appropriately choosing $c>0$, we may ensure that:
\begin{enumerate}
\item 
$\P(\abs{T x} \leq c \sqrt{m/n}) \leq p_{\text{out}} := \frac{1}{2} e^{-3m}$ for any $x \notin B_2^n$.
\item 
$\P(\abs{T x} \leq \sqrt{2} \sqrt{m/n}) \geq p_{\text{in}} := \frac{1}{2}$ for any $x \in B_2^n$. 
\end{enumerate}
Indeed, the first estimate is ensured by Lemma \ref{lem:small-ball}, while the second one follows by simply noting that $\E \abs{T x}^2 = \frac{m}{n} \abs{x}^2$ and applying the Markov--Chebyshev inequality (or by invoking the Johnson--Lindenstrauss Lemma \ref{lem:JL}, but this is actually unnecessary). 
The assertion then follows by Proposition \ref{prop:part1-prob} with $C = \sqrt{2} / c$, $L_T \equiv L/C$ and $S_T \equiv L$. 
\end{proof}

\subsection{Running The Program for Ellipsoids}

Running (the regular version of) The Program of Section \ref{sec:full-program}, we can finally obtain:

\begin{thm}[Generalized Regular Dual Sudakov For Ellipsoids] \label{thm:RegularSudakovEllipsoids}
For any origin-symmetric log-concave measure $\mu$ on $\Real^n$ and any (origin-symmetric) ellipsoid $\Eps \subset \Real^n$, we have:
\[
M(Z_p(\mu) , t m_1(\mu, \Eps) \Eps) \leq \exp \brac{C \brac{ \frac{p}{t^2} + \frac{p}{t} } } \;\;\; \forall p \geq 1 \;\; \forall t > 0 .
\]
\end{thm}
\begin{proof}
Since the expression on the left-hand-side is invariant under simultaneously applying a linear transformation to $\mu$ and $\Eps$, it is enough to establish it for the case that $\Eps = B_2^n$. Since this expression is also invariant under scaling $\mu$, we may assume that $m_1(\mu,B_2^n) = 1$. 

Given $p \geq 1$ and $t > 0$, we run The Program for $K = B_2^n$, with $\L_m$ consisting of (centered) Euclidean balls in $\Real^m$. Corollary \ref{cor:part1-ellipsoids} applied with $q=1$ verifies Part 1' of The Program regarding Massive Separation Dimension Reduction (with say $D=1$, $B=1$, $q_m=3$, $A = 5/4$ and $R \leq C'$). Part 2 of The Program regarding Weak Generalized Regular Dual Sudakov, with parameters $q_m = 3$ and $\varphi_t(x) = C'' \max((x/t)^{2/3}, \sqrt{x} / t)$, is established by Theorem \ref{thm:part2-ellipsoids} in conjunction with Lemma \ref{lem:Guedon} (which implies that $m_3(\nu,L) \simeq I_1(\nu,L)$ for all $\nu \in \M_m$ and $L \in \L_m$). 
Since $\varphi_t^{-1}(y) \simeq \min(y^{3/2} t, y^2 t^2)$, we have $\varphi_t^{-1}(1/(4A)) \simeq \min(t,t^2)$, and Theorem \ref{thm:full-program} yields the asserted estimate. 
\end{proof}

\section{Pure Measures} \label{sec:pure}

In this section, we introduce the class of pure log-concave probability measures, and study their properties. 

\subsection{Definitions}
Recall that the isotropic constant of a probability measure $\mu$ on $\Real^n$ having log-concave density $f_\mu$ is defined as the following affine-invariant quantity:
\begin{equation} \label{eq:Lmu}
L_\mu := (\max f_\mu)^\frac{1}{n} (\det \; \Cov(\mu))^{\frac{1}{2n}} ~. 
\end{equation}
It is well-known (e.g \cite{Milman-Pajor-LK,Klartag-Psi2}) that $L_\mu \geq c$ for some universal constant $c > 0$. See Bourgain \cite{BourgainMaximalFunctionsOnConvexBodies,Bourgain-LK}, Milman--Pajor \cite{Milman-Pajor-LK}, Ball \cite{Ball-PhD} and Brazitikos--Giannopoulos--Valettas--Vritsiou \cite{GreekBook}
for background on the yet unresolved Slicing Problem, which is concerned with obtaining a dimension independent upper-bound on $L_\mu$. The current best-known estimate $L_\mu \leq C n^{1/4}$ is due to B. Klartag \cite{KlartagPerturbationsWithBoundedLK}, who improved the previous estimate $L_\mu \leq C n^{1/4} \log(1+n)$ of J. Bourgain \cite{Bourgain-LK} (proven when $\mu$ is the uniform measure on an origin-symmetric convex body, but valid for general log-concave probability measures, see \cite{Ball-PhD, KlartagPerturbationsWithBoundedLK}); see also Klartag--Milman \cite{KlartagEMilmanLowerBoundsOnZp} and Vritsiou \cite{Vritsiou-ExtendingKM} for subsequent refinements.

The following key estimate, which plays a fundamental role in previous groundbreaking works of Paouris \cite{Paouris-IsotropicTail,PaourisSmallBall} and Klartag \cite{Klartag-Psi2}, relates between $\abs{Z_n(\mu)}$ and $L_\mu$  (see e.g. the proof of \cite[Theorem 2.1]{EMilman-IsotropicMeanWidth}):
\begin{thm}[Paouris, Klartag] \label{thm:Zn}
Let $\mu$ denote a log-concave probability measure on $\Real^n$ with barycenter at the origin. Then:
\[
\volrad(Z_n(\mu)) \simeq \frac{\sqrt{n}}{L_\mu} \vrad(Z_2(\mu)) = \sqrt{n} \; \frac{\det \; \Cov(\mu)^{\frac{1}{2n}}}{L_\mu} = \frac{\sqrt{n}}{\max f_\mu^{1/n}} ~.
\]
In other words, if $\mu$ is isotropic then $\abs{Z_n(\mu)}^{1/n} \simeq \frac{1}{L_\mu}$. 
\end{thm}

\begin{defn}[$h$-pure measure]
Let $\mu$ denote a probability measure on $\Real^n$ with barycenter at the origin. We will say that $\mu$ is $h$-pure ($h=1,\ldots,n$), with constants $(A,B)$, if the following two conditions hold:
\begin{enumerate}
\item $Z_n(\mu) \supset \frac{1}{A} \sqrt{h} Z_2(\mu)$. 
\item For all $E \in G_{n,m}$ with $m = h,\ldots,n$, we have that $L_{\pi_E \mu} \leq B$. 
\end{enumerate}
When $\mu$ is $h$-pure with some universally bounded constants $A,B \leq C < \infty$, we will simply say that $\mu$ is $h$-pure (with implicitly bounded constants). 
\end{defn}

Note that in the second condition, only the marginals of $\mu$ of dimension not smaller than $h$ are taken into account. For example, if $\mu$ is isotropic log-concave and $Z_n(\mu) \supset \frac{1}{A} \sqrt{n} B_2^n$, it follows from Theorem \ref{thm:Zn} that $L_\mu \leq C A$, and hence $\mu$ is $n$-pure (with constants $(A,CA)$). On the other hand, if all marginals of $\mu$ (of arbitrary dimension) have isotropic constant bounded by a universal constant $C>0$, since $Z_n(\mu) \supset Z_2(\mu)$ (if $n \geq 2$, and up to a constant otherwise), we see that $\mu$ is $1$-pure (with constants $(1,C)$). The Slicing Problem may be equivalently restated as asking whether all log-concave probability measures on $\Real^n$ ($n \geq 2$) are $1$-pure with constants $(1,C)$, for some universal constant $C>0$ independent of $n$.  

\medskip
The following is immediate from the definition:
\begin{lem}
If $\mu$ is both $h_1$-pure and $h_2$-pure with $1 \leq h_1 < h_2 \leq n$, then it is also $h$-pure for all $h=h_1,\ldots,h_2$. 
\end{lem}

\subsection{Families of Pure Measures}

We now provide several useful examples of families of log-concave measures which are pure. For simplicity, we restrict our attention to origin-symmetric measures. 
Recall that a measure is called unconditional if it is invariant under reflections with respect to all coordinate hyperplanes. A probability measure $\mu$ on $\Real^n$ is called $\Psi_2$ or sub-Gaussian if $Z_p(\mu) \subset C \sqrt{p} Z_2(\mu)$ for some universal constant $C \geq 1$ and all $p \geq 2$. It is called super-Gaussian if $Z_p(\mu) \supset c \sqrt{p} Z_2(\mu)$ for some universal constant $c > 0$ and all $p \in [2,n]$. It is immediate to verify (see e.g. \cite{GuedonEMilmanInterpolating}) that if $\tilde{\mu}$ is an isotropic origin-symmetric log-concave probability measure and $\gamma_n$ is the standard Gaussian measure, then the convolved measure $\mu = \tilde{\mu} \ast \gamma_n$ is log-concave and super-Gaussian. Finally, a convex body $K$ is called $2$-convex (with constant $\alpha > 0$) if $1 - \norm{\frac{x+y}{2}}_K \geq \alpha \eps^2$ for all $\norm{x}_K,\norm{y}_K \leq 1$ with $\norm{x-y}_K \geq \eps > 0$.

\begin{prop}
The following families of log-concave measures are $n$-pure:
\begin{enumerate}
\item Super-Gaussian measures. 
\item Uniform measures on $2$-convex bodies (with fixed $2$-convexity constant $\alpha>0$).
\end{enumerate}
\end{prop}
\begin{proof}
If $\mu$ is super-Gaussian we have by definition that $Z_n(\mu) \supset c \sqrt{n} Z_2(\mu)$. If  $\mu$ is the uniform measure on a $2$-convex body $K$, it was shown in \cite{KlartagEMilman-2-Convex} that $Z_n(\mu) \supset c_\alpha \sqrt{n} Z_2(\mu)$ for some constant $c_\alpha > 0$ (depending only on $\alpha$, which we assume fixed). In either case, we deduce by Theorem \ref{thm:Zn} that $L_\mu \leq C$, establishing both properties of an $n$-pure measure.
\end{proof}

\begin{prop} \label{prop:pure-1}
The following families of log-concave measures are $1$-pure:
\begin{enumerate}
\item Super-Gaussian measures. 
\item Sub-Gaussian ($\Psi_2$) measures. 
\item Unconditional measures. 
\end{enumerate}
Furthermore, the class of $1$-pure measures is closed under taking marginals. 
\end{prop}
\begin{proof}
We may assume that all measures in question are isotropic. 

\begin{enumerate}
\item
If $\mu$ is super-Gaussian, we have for all $E \in G_{n,m}$:
\[
Z_m(\pi_E \mu) = P_E Z_m(\mu) \supset c \sqrt{m} B_2(E) ,
\]
so that $\frac{1}{L_{\pi_E \mu}} \simeq \abs{Z_m(\pi_E \mu)}^{1/m} \geq c' > 0$, as asserted. 
\item 
Similarly, if $\mu$ is $\Psi_2$ then for all $E \in G_{n,m}$ and $p \geq 2$:
\[
Z_p(\pi_E \mu) = P_E Z_p(\mu) \subset C \sqrt{p} B_2(E) ,
\]
confirming that $\pi_E \mu$ is also $\Psi_2$ (with the same universal bound $C$ on its $\Psi_2$ constant). It is well-known \cite{Bourgain-Psi-2-Bodies,KlartagEMilmanLowerBoundsOnZp} that a $\Psi_2$ measure has bounded isotropic constant, confirming the assertion in this case. 
\item 
If $\mu$ is an isotropic unconditional measure and $\chi$ is the uniform measure on the cube $[-1,1]^n$, it was noted in \cite[p. 2829]{DafnisGiannopoulosTsolomitis-RandomPolytopes} following Lata{\l}a that $Z_p(\mu) \supset c Z_p(\chi)$ for all $p \geq 1$. It follows as before that if $E \in G_{n,m}$ then:
\[
Z_m(\pi_E \mu) = P_E Z_m(\mu) \supset c P_E Z_m(\chi) = c Z_m(\pi_E \chi) .
\]
Taking volumes, we deduce:
\[
 \frac{1}{L_{\pi_E \mu}} \simeq \abs{Z_m(\pi_E \mu)}^{1/m} \geq c \abs{Z_m(\pi_E \chi)}^{1/m} \simeq \frac{1}{L_{\pi_E \chi}} .
\]
But since $\chi$ is a $\Psi_2$ measure, we know that $L_{\pi_E \chi}$ is universally bounded above, establishing (3). 
\end{enumerate}
The closure under marginals is immediate from the definition and the fact that $Z_m(\mu) \supset c Z_2(\mu)$ for all $m \geq 1$ by (\ref{eq:Zpq}).  
\end{proof}

\begin{rem}
It was shown by Paouris in \cite{Paouris-MarginalsOfProducts} that product measures (having arbitrarily many factors) of sub-Gaussian or super-Gaussian log-concave measures are $1$-pure. 
\end{rem}

\begin{rem}
A well-known argument due to V.~Milman involving the M-position \cite{AGA-Book-I,Pisier-Book}, in combination with K.~Ball's observation that the isotropic position is an M-position if the isotropic constant is bounded \cite{Ball-PhD}, shows that if $\mu$ is an origin-symmetric isotropic log-concave probability measure with $L_\mu \leq C$, then with high-probability, a random marginal $\pi_F \mu$ with $F \in G_{n,n/2}$ is $n/2$-pure with universal constants $(A,B)$ depending solely on $C$. 
Moreover, with high-probability, a random marginal $\pi_F \mu$ with $F \in G_{n,n/4}$ is super-Gaussian, and therefore $h$-pure for all $h=1,\ldots,n/4$ with universal constants $(A,B)$ depending solely on $C$; we briefly sketch the proof. Let $ {\bar{sg}}(\nu)$ denote the supergaussian constant of a probability measure $\nu$ in $\Real^m$, i.e.  the minimum $t>0$ such that $ \frac{1}{ t} \sqrt{p} Z_{2}(\nu) \subseteq Z_{p}(\nu)$ for all $2\leq p \leq m$. It is easy to check that if $\mu_{s}$ denotes the conditioning of $\mu$ onto $s \sqrt{n} B_2^n$ for a suitably chosen constant $s \simeq 1$, then ${\bar{sg}}(\pi_{F}(\mu))\leq C {\bar{sg}}(\pi_{F}(\mu_{s}))$ for all subspaces $F$,  since $ Z_{p} (\mu_{s}) \subset e \; Z_{p} (\mu)$ for all $2\leq p \leq n$. Moreover, one may check (e.g. by using \cite[Proposition 5.5]{PaourisSmallBall}) that  $M^*(Z_{p}(\mu_{s}))\leq  C' \sqrt{p}$
and (since $L_{\mu_{s}} \simeq L_{\mu} \leq C$) $\vrad(Z_{p}(\mu_{s})) \geq c \sqrt{p}$, for all $2\leq p \leq n$.
Applying \cite[Proposition 3.1]{Klartag-Vershynin} and the Bourgain--Milman reverse Santal\'o inequality \cite[Theorem 8.2.2]{AGA-Book-I}, it follows that  for $k\leq n/4$, with probability at least $1- e^{ -c n}$ over $F\in G_{n,k}$ (with respect to the corresponding Haar probability measure), one has that the inradius of $P_{F}Z_{p}(\mu_s)$ is at least $c^{\prime} \sqrt{p}$. Using this fact for all $p=2^{m}$ with $m=1, \cdots ,[\log_{2}{n}]$, and invoking the identity $P_{F} Z_{p}(\mu_{s}) = Z_{p} ( \pi_{F} (\mu_{s}))$, an immediate application of the union bound yields that with high-probability on $F\in G_{n,k}$:
\[
c\sqrt{p} Z_{2} (\pi_{F} (\mu_{s})) \subset Z_{p} (\pi_{F} (\mu_{s})) \;\;\; \forall p=2^{m} ~,~ m=1, \cdots , [\log_{2} {n}]. 
\]
This shows that $ {\bar{sg}}(\pi_{F}(\mu)) \leq C_1 {\bar{sg}}(\pi_{F}(\mu_{s}))\leq C_2$ with  probability at least $1- e^{-c n}$ over $F \in G_{n,k}$, completing the proof. 

\end{rem}

\subsection{Properties of Pure Measures}

Given a convex body $K \subset \Real^n$ and $m=1,\ldots,n$, we use the following notation:
\begin{align*}
v_m^-(K) & := \inf \set{ \vrad(P_E (K)) ; E \in G_{n,m} }, \\
e_m(B_2^n,K) & := \inf \set{t > 0 \; ;\; N(B_2^n, tK) \leq 2^m} . 
\end{align*}
We will need the following crucial estimate on the regularity of dual covering numbers of pure isotropic log-concave measures, essentially established by Giannopoulos and Milman in \cite{GiannopoulosEMilman-IsotropicM}:

\begin{thm} \label{thm:pure-regular}
Let $\mu$ denote a $h$-pure isotropic log-concave probability measure (with constants $(A,B)$). Then for all $k=1,\ldots,n$:
\begin{equation} \label{eq:pure-prop1}
 v_k^-(Z_n(\mu)) \geq \max \brac{\frac{1}{A} \sqrt{h} , \frac{c}{B} \sqrt{k}} .
\end{equation} 
Furthermore, for all $k=1,\ldots,n$ we have:
\begin{equation} \label{eq:pure-prop2}
e_k(B_2^n,Z_n(\mu)) \leq \min \brac{ \frac{A}{\sqrt{h}} , C_{A,B} \frac{1}{\sqrt{k}} \frac{n}{k} \log( e + \frac{n}{k}) } ,
\end{equation}
or equivalently, we have for all $t > 0$:
\begin{equation}  \label{eq:pure-prop3}
N(\sqrt{n} B_2^n , t Z_n(\mu)) \leq \begin{cases}  \exp \brac{C'_{A,B} \; n \brac{\frac{\log(e+t)}{t}}^{\frac{2}{3}}}  &  t \leq A \sqrt{n/h} \\ 
1 & \text{otherwise} \end{cases} .
\end{equation}
\end{thm}
\begin{proof}[Proof Sketch]
Since $\mu$ is assumed isotropic we have $Z_2(\mu) = B_2^n$. 
The case $h=1$ appears explicitly in \cite[Lemma 12 and Theorem 16]{GiannopoulosEMilman-IsotropicM}. For the general case, an inspection of the proof of \cite[Theorem 16]{GiannopoulosEMilman-IsotropicM} reveals that the only ingredient required to obtain an estimate on $e_k(B_2^n,Z_n(\mu))$ is a lower bound on $v_m^-(Z_n(\mu))$ for $m=1,\ldots, k$.

When $m \leq h$, we may simply use $Z_n(\mu) \supset \frac{1}{A} \sqrt{h} B_2^n$ and conclude that:
\[
 v_m^-(Z_n(\mu)) \geq \frac{1}{A} \sqrt{h} .
\]
When $m = h,\ldots,n$, \cite[Lemma 12]{GiannopoulosEMilman-IsotropicM} ensures that:
\[
 v_m^-(Z_n(\mu)) \geq \frac{c}{\sup \set{ L_{\pi_E \mu} ; E \in G_{n,m} }} \sqrt{m} ,
 \]
for some universal constant $c > 0$, and so we see that one only needs to control the isotropic constants of marginals of $\mu$ of dimension not smaller than $h$. 
Combining these two estimates, (\ref{eq:pure-prop1}) follows for a $h$-pure isotropic measure with constants $(A,B)$. 

Now, according to \cite[Corollary 9 and Remark 6]{GiannopoulosEMilman-IsotropicM}, one has for any $\alpha > 0$:
\[
e_k(B_2^n,K) \leq C_\alpha \sup_{m=1,\ldots,k}  \brac{\frac{m}{k}}^{\alpha} \frac{n}{m} \log\Big(e + \frac{n}{m}\Big) \frac{1}{v_{m}^-(K)}.
\]
Applying this to $K = Z_n(\mu)$ (with, say, $\alpha = 2$), and plugging the estimate (\ref{eq:pure-prop1}) on $v_{m}^-(K)$, we obtain: 
\[
e_k(B_2^n, Z_n(\mu)) \leq C_{A,B} \frac{1}{\sqrt{\max(h,k)}} \frac{n}{k} \log( e + \frac{n}{k}) 
\]
Combining this with the trivial estimate:
\[
e_k(B_2^n,Z_n(\mu)) \leq \frac{A}{\sqrt{h}} 
\]
(since $Z_n(\mu) \supset \frac{1}{A} \sqrt{h} B_2^n$), the asserted (\ref{eq:pure-prop2}) follows (note that we replaced $\max(h,k)$ by the looser $k$ since we do not care here about the dependence of $C_{A,B}$ on $(A,B)$).  The equivalent (\ref{eq:pure-prop3}) is obtained in the range $t \geq 1$ by direct inspection of (\ref{eq:pure-prop2}), and extended to all $t > 0$ by Lemma \ref{lem:extend} after adjusting the constant $C_{A,B}'$. 
\end{proof}

\section{Part 2 - Weak Generalized Dual Sudakov: Pure Measures and Regular Small-Diameter Bodies} \label{sec:part2}

It is naturally of interest to establish the Weak Generalized Dual Sudakov estimate for general (say origin-symmetric) log-concave measures $\mu$ and convex bodies $K$. Unfortunately, we have not been able to accomplish this in that generality. In this section, we establish a Weak Generalized Dual Sudakov estimate when the log-concave measure $\mu$ is assumed to be $h$-pure, or when $K$ is assumed to have $\alpha$-regular small-diameter (defined below). 

\subsection{Part 2 for Pure Measures}

\begin{thm}  \label{thm:part2-pure}
Let $\mu$ be a $h$-pure log-concave measure on $\Real^n$ (with constants $(A,B)$), and let $p \in [1,n]$. Then for any $q > 0$ and star-body $L \subset \Real^n$:
\[
M\brac{Z_p(\mu) , t \sqrt{\frac{p}{n}} m_q(\mu,L) L} \leq \exp \brac{1 + q + C_{A,B} n \brac{\frac{\log(e+t)}{t}}^{\frac{1}{3}} } \;\;\; \forall t > 0. \]
In particular, if $\mu(L) \geq \exp(-p)$ with $p \in [1,n]$ then:
\[
M(Z_p(\mu) , L) \leq \exp( C'_{A,B} n^{5/6} p^{1/6} \log^{1/3}(1+n/p)) .
\]
\end{thm}

This will follow from Part 3 of The Program together with the following:

\begin{thm} \label{thm:Zpn}
Let $\mu$ be a $h$-pure log-concave measure (with constants $(A,B)$) and let $p \in [1,n]$. Then:
\[
N\brac{Z_p(\mu) , t \sqrt{\frac{p}{n}} Z_n(\mu)} \leq \exp \brac{ C_{A,B} n \brac{\frac{\log(e+t)}{t}}^{\frac{1}{2}}} \;\;\; \forall t > 0 .
\]
\end{thm}
\begin{proof}
Since the statement is invariant under linear transformations, we may assume that $\mu$ is isotropic. 
By the triangle inequality for covering numbers, we have for all $s  > 0$:
\[
N\brac{Z_p(\mu) , t \sqrt{\frac{p}{n}} Z_n(\mu)} \leq N\brac{Z_p(\mu) , \frac{t}{s} \sqrt{p} B_2^n} N\brac{\sqrt{n} B_2^n , s Z_n(\mu)} .   
\]

Since $I_1(\mu,B_2^n) \simeq \sqrt{n}$ for isotropic $\mu$, the following estimate is a particular case of the Generalized Dual Sudakov estimate for ellipsoids we obtained in Theorem \ref{thm:RegularSudakovEllipsoids} for general (not necessarily isotropic) origin-symmetric log-concave measures:
\[
N(Z_p(\mu) , r \sqrt{p} B_2^n) \leq \exp\brac{ C_2 \frac{n}{r^2}  + C_3 \frac{\sqrt{n} \sqrt{p}}{r} } \;\;\; \forall r > 0.
\]
For isotropic log-concave measures, this estimate was first established in \cite[Proposition 5.1]{GPV-ImprovedPsi2} -- see Subsection \ref{subsec:conclude-ell} for more details.

Next, by Theorem \ref{thm:pure-regular} on the regularity of dual covering numbers for pure measures, we have:
\[
N(\sqrt{n} B_2^n , s Z_n(\mu))  \leq \exp \brac{ C_{A,B} n \brac{\frac{\log(e+s)}{s}}^{\frac{2}{3}} } \;\;\; \forall s >0 . 
\]
Setting $r = t/s$ above and combining both estimates, we obtain for all $s,t > 0$:
\[
N(Z_p(\mu) , t \sqrt{p/n} Z_n(\mu)) \leq \exp \brac{ C_2 \frac{n}{t^2} s^2 + C_3 \frac{\sqrt{n}{\sqrt{p}}}{t} s + C_{A,B} n  \brac{\frac{\log(e+s)}{s}}^{\frac{2}{3}} } .
\]
Optimizing on $s$, we set $s = t^{3/4} \log^{1/4}(e+t)$, yielding:
\[
N(Z_p(\mu) , t \sqrt{p/n} Z_n(\mu)) \leq \exp \brac{ C'_{A,B} n \brac{\frac{\log(e+t)}{t}}^{1/2} + C_3 \sqrt{n} \sqrt{p}  \brac{\frac{\log(e+t)}{t}}^{1/4}} \;\;\; \forall t > 0 . \] 
However, note that since $Z_p(\mu) \subset Z_n(\mu)$, the left-hand-side is exactly $1$ for all $t \geq \sqrt{n/p}$, and that in the non-trivial range $t \in (0 , \sqrt{n/p}]$, the first term on the right-hand-side always dominates the second one. Adjusting constants, the assertion is thus established. 
\end{proof}

\begin{proof}[Proof of Theorem \ref{thm:part2-pure}]
By the triangle inequality for packing numbers, we have for every $s > 0$:
\[
M\brac{Z_p(\mu) , t \sqrt{\frac{p}{n}} m_q(\mu,L) L} \leq N\brac{Z_p(\mu) , s \sqrt{\frac{p}{n}} Z_n(\mu)} M\brac{Z_n(\mu) , \frac{t}{s} m_q(\mu,L) L} . 
\]
Invoking Theorems \ref{thm:Zpn} and \ref{thm:part3} to estimate the terms on the right-hand-side, we obtain:
\[
M\brac{Z_p(\mu) , t \sqrt{\frac{p}{n}} m_q(\mu,L) L} \leq \exp \brac{ C_{A,B} n \brac{\frac{\log(e+s)}{s}}^{1/2} + 1 + q + C' n \frac{s}{t}} .  
\]
Optimizing on $s > 0$, we set $s := t^{2/3} \log^{1/3} (e+t)$. Adjusting constants, the assertion is established. 
\end{proof}

\subsection{Part 2 for Regular Small-Diameter Bodies}

The only property we will require for the ensuing proof is encapsulated in the following:

\begin{dfn*}[Regular Small-Diameter]
An origin-symmetric convex body $K \subset \Real^n$ is called $\alpha$-regular ($\alpha \in (0,2]$) small-diameter (with constant $R\geq 1$) if there exists $T \in GL_n$ so that, denoting $K_0 = T(K)$:
\begin{enumerate}
\item $K_0 \subset R  B_2^n$ (``small-diameter").
\item $N(B_2^n,t K_0) \leq \exp( n / t^\alpha)$ for all $t > 0$ (``$\alpha$-regular"). 
\end{enumerate}
\end{dfn*}
Note that if $K$ is $\alpha$-regular small-diameter (with constant $R$), then it is also $\beta$-regular small-diameter (with constant depending on $R$ and $\beta$) for all $\beta \in (0,\alpha]$ (as follows for instance from Lemma \ref{lem:extend}). Also note that simple examples such as the unit-cube show that one cannot expect a general origin-symmetric convex  $K \subset \Real^n$ to be $\alpha$-regular small-diameter with $R \leq C$ and $\alpha > 2$ (independently of $n$).

\subsubsection{Examples of Regular Small-Diameter Bodies}

\begin{prop} \label{prop:psi2}
Assume that $K \subset \Real^n$ is an origin-symmetric convex body so that $\lambda_K$, the uniform measure on $K$, is $h$-pure (with constants $(A,B)$). Assume in addition that $K \subset D \sqrt{n} Z_2(\lambda_K)$. Then $K$ is $\alpha$-regular small-diameter with constant $R_{A,B,\alpha} D$ for all $\alpha \in (0,2/3)$.

In particular, the following families are $\frac{1}{2}$-regular small-diameter:
\begin{enumerate}
\item Sub-Gaussian ($\Psi_2$) Convex-Bodies are $\frac{1}{2}$-regular small-diameter with constant $R \leq C$.
\item Unconditional Convex Bodies $K$ satisfying $K \subset D \sqrt{n} Z_2(\lambda_K)$ are $\frac{1}{2}$-regular small-diameter with constant $R \leq C D$. 
\end{enumerate}
\end{prop}
\begin{proof}
We may assume that $\lambda_K$ is isotropic, so that $Z_2(\lambda_K) = B_2^n$. Since $Z_n(\lambda_K) \simeq K$ by origin-symmetry, if we define $K_1 = K / \sqrt{n}$, Theorem \ref{thm:pure-regular} ensures that:
\[
N(B_2^n, t K_1) \leq \exp(C_{A,B,\alpha} n / t^\alpha) \;\;\; \forall t > 0 ,
\]
for any $\alpha \in (0,2/3)$, while we are given that $K_1 \subset D B_2^n$. It follows that $K$ is $\alpha$-regular with constant $R = C_{A,B,\alpha}^{1/\alpha} D$. \\
In particular, if $K$ is sub-Gaussian, Proposition \ref{prop:pure-1} ensures that $\lambda_K$ is $1$-pure, while it is also well-known (e.g. \cite{PaourisPsi2Behaviour}) that $K  \simeq Z_n(\lambda_K) \subset D \sqrt{n} Z_2(\lambda_K)$ where $D$ is the $\Psi_2$ constant of $K$, which is assumed to be bounded by a universal constant. In addition, Proposition \ref{prop:pure-1} ensures that $\lambda_K$ is $1$-pure if $K$ is unconditional, and so if in addition $K \subset D \sqrt{n} Z_2(\lambda_K)$, then it is $\frac{1}{2}$-regular with constant $R = C D$.
\end{proof}

To describe another important class of regular small-diameter bodies, recall that the (Gaussian) type-2 constant of a normed space $(X,\norm{\cdot})$ over $\Real$,
denoted $T_2(X)$, is the minimal $T>0$ for which:
\[
\brac{\E \snorm{\sum_{i=1}^m G_i x_i}^2}^{\frac{1}{2}} \leq T \brac{\sum_{i=1}^m \norm{x_i}^2}^{\frac{1}{2}}
\]
for any $m \geq 1$ and any $x_1,\ldots,x_m \in X$, where
$G_1,\ldots,G_m$ denote independent real-valued standard Gaussian
random variables. 
We will often identify between a normed space and its unit-ball, and given an origin-symmetric convex body $K \subset \Real^n$, 
refer to the type-2 constant $T_2(X_K)$ of the normed space $X_K$ whose unit-ball is $K$. 
We will not distinguish between the Gaussian and the Rademacher type-2 constants, since it is well known that the former constant is
always majorized by the latter one (e.g. \cite{Milman-Schechtman-Book}), and all our results will involve
upper bounds in terms of the Gaussian type-2 constant. 

Note that a Hilbert-space has type-2 constant exactly $1$. It is also well-known (e.g. \cite{Milman-Schechtman-Book}) that subspaces of $L_p$ for $p \geq 2$ have type-2 constant of the order of $\sqrt{p}$. In a finite dimensional setting, it is clear by John's theorem that $T_2(X_K) \leq \sqrt{n}$ for all origin-symmetric $K \subset \Real^n$. Since $\ell_\infty^n$ is isomorphic to a subspace of $L_{\log n}$, it similarly follows that $T_2(\ell_\infty^n) \leq C \sqrt{\log n}$, and in fact this is the correct order. 
\begin{prop} \label{prop:type-2}
Every origin-symmetric convex body $K \subset \Real^n$ is $2$-regular small-diameter with constant $C T_2(X_K)$, for some universal constant $C \geq 1$. 
\end{prop}
\begin{proof}
It was shown by W.~J.~ Davis, V.~Milman and N.~Tomczak-Jaegermann \cite{Davis-etal-Lemma} using operator-theoretic notation, and in \cite[Corollary 3.5]{EMilman-DualMixedVolumes} using a geometric argument, that when $B_2^n$ is the minimal volume ellipsoid containing $K$ (the Lowner position), then:
\[
M(K) \leq M_2(K) := \brac{\int_{S^{n-1}} \norm{\theta}_K^2 d\sigma(\theta)}^{1/2} \leq T_2(X_K) .
\]
Setting $K_0 := R K$ with $R = C T_2(X_K)$ for an appropriate constant $C \geq 1$, we have that $M(K_0) \leq \frac{1}{C}$, and hence by the Dual Sudakov Minoration:
\[
N(B_2^n, t K_0) \leq \exp(n / t^2) \;\;\; \forall t > 0 .
\]
Since $K_0 \subset R B_2^n$, the assertion is established. 
\end{proof}

\begin{rem} \label{rem:lq}
Applying Proposition \ref{prop:psi2}, we may conclude that the unit-balls $B_q^n$ of $\ell_q^n$, which for all $q \in [2,\infty]$ are both $\Psi_2$ (see e.g. \cite{BGMN}) and small-diameter unconditional, are $1/2$-regular with uniformly bounded universal constant $C > 0$. Note that by the previous remarks, we cannot get a uniform estimate for the type-2 constant of $\ell_q^n$ in the range $q \in [2,\infty]$, and the precise covering estimates due to C.~Sch\"utt \cite[Theorem 1]{Schutt-Entropy-numbers} imply that $\ell_\infty^n$ is not 2-regular small-diameter with dimension-independent constant. However,  Sch\"utt's estimates yield that for all $q \in [2,\infty]$ and $\eps > 0$, $B_q^n$ is in fact $(2-\eps)$-regular small-diameter with constant $C_\eps>0$ depending only on $\eps > 0$. For simplicity, as this is not crucial 
for any of our ensuing estimates, we will only use below that they are all $1$-regular small-diameter with a uniformly bounded universal constant $C>0$ for all $q \in [2,\infty]$. 
\end{rem}

\subsubsection{Weak Generalized Regular Dual Sudakov}

\begin{thm} \label{thm:part2-smalldiam}
Let $\mu$ denote an origin-symmetric log-concave probability measure on $\Real^n$, and let $K \subset \Real^n$ denote an $\alpha$-regular small-diameter convex body (with constants $\alpha \in (0,2]$ and $R\geq 1$). Then for any $p \in [1,n]$:
\begin{equation} \label{eq:type-2-gen}
N(Z_p(\mu) , t m_1(\mu,  K) K) \leq \exp \brac{ C (n (R / t)^\alpha)^{\frac{2}{2+\alpha}} p^{\frac{\alpha}{2+\alpha}} + C (n (R / t)^\alpha)^{\frac{1}{1+\alpha}} p^{\frac{\alpha}{1+\alpha}}}  \;\;\; \forall t > 0 . 
\end{equation}
In particular:
\[
N(Z_p(\mu) , T_2(X_K) m_1(\mu, K) K) \leq \exp \brac{ C' \sqrt{n} \sqrt{p}}  ,
\]
and for all $q \in [2,\infty]$:
\[
N(Z_p(\mu) , t m_1(\mu,B_q^n)  B_q^n) \leq \exp \brac{ C'' \frac{p^{1/3} n^{2/3}}{t^{2/3}} + C'' \frac{\sqrt{p} \sqrt{n}}{\sqrt{t}} } \;\;\; \forall t > 0 . 
\]
\end{thm}
\begin{proof}
Since the statement is clearly invariant under applying (non-degenerate) linear transformations on both $\mu$ and $K$, we may and will assume that $K \subset R B_2^n$ and:
\[
N(B_2^n, t K) \leq \exp(n / t^\alpha) \;\;\; \forall t > 0 .
\]
The small-diameter property ensures that:
\[
I_1(\mu) = \int \abs{x} d\mu(x) \leq R \int \norm{x}_K d\mu(x) = R I_1(\mu,K) . 
\]
Together with the Generalized Regular Dual Sudakov estimate for ellipsoids (Theorem \ref{thm:RegularSudakovEllipsoids}), we obtain for any $s > 0$:
\begin{align*}
N(Z_p(\mu) , t I_1(\mu,K) K) & \leq N(Z_p(\mu) , s I_1(\mu) B_2^n) N(I_1(\mu) B_2^n , (t/s) I_1(\mu,K) K ) \\
& \leq N(Z_p(\mu) , s I_1(\mu) B_2^n) N(B_2^n , t/(R s) K ) \\
& \leq \exp\brac{ C  \brac{\frac{p}{s^2} + \frac{p}{s}} + n \frac{R^\alpha s^{\alpha}}{t^{\alpha}} } . 
\end{align*}
Optimizing on $s > 0$, we set $s = C' \max(s_{1+\alpha}, s_{2+\alpha})$ where $s_\beta := \brac{\frac{p}{n} \brac{\frac{t}{R}}^\alpha}^{1/\beta}$. Recalling that $I_1(\mu,K) \simeq m_1(\mu,K)$ by Lemma \ref{lem:Guedon}, we obtain the first assertion. Applying the first part with $\alpha=2$, $t = C'' T_2(X_K)$ and invoking Proposition \ref{prop:type-2}, the second assertion follows after an adjustment of constants. The last assertion follows in view of Remark \ref{rem:lq}. 
\end{proof}

Note that by Lemma \ref{lem:ZnHuge} we always have:
\[
m_1(\mu,Z_n(\mu)) \simeq I_1(\mu,Z_n(\mu)) \leq C ,
\]
and in addition $N(Z_p(\mu) , t Z_n(\mu)) = 1$ for all $p \in [1,n]$ and $t \geq 1$. Consequently, when $K = Z_n(\mu)$, only the first term on the right-hand-side of (\ref{eq:type-2-gen}) is relevant, and we obtain:

\begin{cor} \label{cor:type-2}
Let $\mu$ denote an origin-symmetric log-concave probability measure on $\Real^n$, so that $Z_n(\mu)$ is $\alpha$-regular small-diameter (with constants $\alpha \in (0,2]$ and $R\geq 1$). Then for $p \in [1,n]$:
\[
N\brac{Z_p(\mu) , t \sqrt{\frac{p}{n}} Z_n(\mu)} \leq \exp \brac{ C n \frac{R^{\frac{2\alpha}{2+\alpha}}}{t^{\frac{2\alpha}{2+\alpha}}} }  \;\;\; \forall t > 0 . 
\]
In particular:
\[
N\brac{Z_p(\mu) , t \sqrt{\frac{p}{n}} Z_n(\mu)} \leq \exp \brac{ C n  \frac{T_2(X_{Z_n(\mu)})}{t} }  \;\;\; \forall t > 0 . 
\]
\end{cor}

\medskip
In analogy with the previous subsection, we deduce:

\begin{thm}
Let $\mu$ denote an origin-symmetric log-concave probability measure on $\Real^n$, so that $Z_n(\mu)$ is $\alpha$-regular small-diameter (with constants $\alpha \in (0,2]$ and $R\geq 1$). Then for $p \in [1,n]$, $q > 0$ and star-body $L \subset \Real^n$:
\[
M\brac{Z_p(\mu) , t \sqrt{\frac{p}{n}} m_q(\mu,L) L}  \leq \exp \brac{1 + q + C  n \brac{\frac{R}{t}}^{\frac{2\alpha}{3\alpha+2}} } \;\;\; \forall t > 0.
\]
In particular:
\[
M(Z_p(\mu) , m_p(\mu, L) L) \leq \exp \brac{ C' \sqrt{T_2(X_{Z_n(\mu)})} \sqrt{n} \sqrt{p}}  .
\]
\end{thm}
\begin{proof}
By the triangle inequality for packing numbers, we have for every $s > 0$:
\[
M\brac{Z_p(\mu) , t \sqrt{\frac{p}{n}} m_q(\mu,L) L}  \leq N\brac{Z_p(\mu) , s  \sqrt{\frac{p}{n}} Z_n(\mu)} M\brac{Z_n(\mu) , \frac{t}{s} m_q(\mu,L) L} . 
\]
Invoking Corollary \ref{cor:type-2} and Theorem \ref{thm:part3} to estimate the terms on the right-hand-side, we obtain:
\[
M\brac{Z_p(\mu) , t \sqrt{\frac{p}{n}} m_q(\mu,L) L}  \leq \exp \brac{ C n \frac{R^{\frac{2\alpha}{2+\alpha}}}{s^{\frac{2\alpha}{2+\alpha}}} + 1 + q + C' n \frac{s}{t}} .  
\]
Optimizing on $s>0$, we set $s := t^{\frac{\alpha+2}{3\alpha+2}} R^{\frac{2\alpha}{3\alpha+2}}$, establishing the assertion after adjustment of constants. The last part follows by Proposition \ref{prop:type-2}.
\end{proof}

\section{Part 1 - Combinatorial Dimension Reduction: Cube} \label{sec:part1-cubes}

In this section, we establish Part 1 of The (full) Program for the case that $K = B_\infty^n$, the $n$-dimensional cube, albeit with $D = C \log(e+n)$ and $R = C \log \log (e+n)$. Contrary to the linear ``One-Sided Johnson--Lindenstrauss" approach that worked well for $K = B_2^n$, we employ a non-linear combinatorial dimension reduction, based on the fundamental work of M.~Rudelson and R.~Vershynin \cite{RudelsonVershynin-CombDim} on the combinatorial dimension, extending the work of Mendelson and Vershynin from \cite{MendelsonVershynin-CombDim}.

\subsection{Part 1 via Cell Content and Combinatorial Dimension}

Denote by $G_{\text{crd}}$ the collection of all $2^n$ coordinate subspaces of $\Real^n$ (of arbitrary dimension $m=0,1,\ldots,n$). 
Given a convex body $K \subset \Real^n$, its cell content $\Sigma(K)$ is defined as:
\[
\Sigma(K) := \sum_{F \in G_{\text{crd}}} \text{number of integer cells contained in $P_F K$} ,
\]
where an integer cell is defined as a unit-cube with integer coordinates, i.e. $x + [0,1]^m$ with $x \in \mathbb{Z}^m$. When $F = \set{0}$, the number of integer cells contained in $P_F K$ is defined to be $1$. The combinatorial dimension $v(K)$ is defined to be:
\[
v(K) := \max \set{ \text{dim}(F) \; ; \; F \in  G_{\text{crd}} \text{ and $P_F K$ contains at least one integer cell} }.
\]

\medskip
Recall that $B_\infty^n  := [-1,1]^n$. The combinatorial information we will require is summarized in the following theorem, which is a particular case of \cite[Theorem 4.2]{RudelsonVershynin-CombDim}:

\begin{thm}[Rudelson--Vershynin] \label{thm:RV}
Let $K \subset \Real^n$ denote a convex body so that $N(K , B_\infty^n) \geq \exp(a n)$, $a > 0$. Then for all $\eps > 0$:
\[
N(K,B_\infty^n) \leq \brac{\Sigma \brac{\frac{C}{\eps} K }}^{M_\eps} ~,~ M_\eps := 4 \log^{\eps}(e + 1/a) . 
\]
\end{thm}
We will also require an additional standard combinatorial lemma (see \cite[Lemma 4.6]{RudelsonVershynin-CombDim}), which may be seen as an integer-valued extension of the Sauer-Shelah lemma:
\begin{lem} \label{lem:SS}
If $K \subset a B_\infty^n$ then:  
\[
\Sigma(K) \leq \brac{\frac{C a n}{v(K)}}^{v(K)} . 
\]
\end{lem}

We can now state:
\begin{thm}[Part 1 for $K = B_\infty^n$ with logarithmic factors] \label{thm:part1-cubes}
Let $\mu$ be an origin-symmetric log-concave measure on $\Real^n$, let $p \in [1,n]$ and $t \geq 1/n$. 
Set:
\[
D := C_1 \log(e+n)  ~,~  R:= C_2 \log \log (e+n) ,
\]
for appropriate universal constants $C_1,C_2 \geq 1$. 
Assume that $M(Z_p(\mu) , t R B_\infty^n) = e^k$ with $\mu(B_\infty^n) \geq \frac{1}{e}$ and $1 \leq k \leq n$. 
Then there exists $F \in G_{\text{crd}}$ of $\text{dim}(F) = m \in [k/D,k]$, so that:
\begin{enumerate}
\item  $M(P_F Z_p(\mu), t P_F B_\infty^n) \geq e^m$ (\textbf{``Partial Separation Dimension Reduction"}).
\item $\pi_F \mu(P_F B_\infty^n) \geq \mu(B_\infty^n) \geq \frac{1}{e}$ (``\textbf{$P_F B_\infty^n$ is sufficiently massive}"). 
\end{enumerate}
\end{thm}
\begin{proof}
We know that:
\[
N\brac{\frac{1}{t R} Z_p(\mu) , B_\infty^n} \geq M(Z_p(\mu) , t R B_\infty^n) = e^k ,
\]
and so by Theorem \ref{thm:RV}, we have for any $\eps > 0$:
\begin{equation} \label{eq:inter1}
k \leq 4 \log^{\eps}\brac{e + \frac{n}{k}} \log \Sigma\brac{ \frac{C}{t R \eps} Z_p(\mu) } .
\end{equation}
By Lemma \ref{lem:IpWish1} and \ref{lem:Guedon}, $\mu(B_\infty^n) \geq 1/e$ implies that:
\[
Z_p(\mu) \subset I_p(\mu,B_\infty^n) B_\infty^n \subset C' p \; m_1(\mu,B_\infty^n) B_\infty^n \subset C' p B_\infty^n ,
\]
and so $\frac{C}{t R \eps} Z_p(\mu) \subset \frac{C'' p}{t R \eps} B_\infty^n$. Applying Lemma \ref{lem:SS}, we deduce that:
\begin{equation} \label{eq:inter2}
\log \Sigma\brac{ \frac{C}{t R \eps} Z_p(\mu) } \leq m_\eps \log \brac{\frac{C_3 p n}{t R \eps m_\eps}  } ~,~ m_\eps := v \brac{\frac{C}{t R \eps} Z_p(\mu)} .
\end{equation}
Setting $\eps = 1 / \log \log (e+n)$ and $C_2 = 8 C$, we ensure by (\ref{eq:inter1}) and (\ref{eq:inter2}) that $m := v (\frac{1}{8t} Z_p(\mu))$ satisfies:
\[
k \leq 4 e m \log\brac{\frac{C_3 p n}{t C_2 m}  }  .
\]
Since $p \in [1,n]$ and $t \geq 1/n$, by appropriately selecting $C_1$ we may ensure that:
\[
m \geq k / D . 
\]
This means that there exists $F \in G_{\text{crd}}$ of $\text{dim}(F) = m \geq k/D$ so that $\frac{1}{8t} P_F Z_p(\mu)$ contains an integer cell. In particular (as $M([0,1] , [0,1/4]) \geq e$):
\[
M(P_F Z_p(\mu) , t P_F B_\infty^n) = M( \frac{1}{8t} P_F Z_p(\mu) , \frac{1}{8} P_F B_\infty^n ) \geq M(\frac{1}{2} P_F B_\infty^n , \frac{1}{8} P_F B_\infty^n) \geq e^m . 
\]
Of course, by decreasing $m$ if necessary, we may also always ensure that $m \leq k$. This concludes the proof of the first assertion. The second assertion is obvious since $\pi_F\mu( P_F B_\infty^n ) = \mu(P_F^{-1} P_F B_\infty^n) \geq \mu(B_\infty^n)$. 
\end{proof}

\subsection{Running The Program For Cubes}

Running The (full) Program, we finally obtain:
\begin{thm}[Generalized Regular Dual Sudakov For Cubes with Logarithmic Terms] \label{thm:RegularSudakovCubes}
For any origin-symmetric log-concave measure $\mu$ on $\Real^n$, we have:
\[
M(Z_p(\mu) , t C \log \log (e+n) m_1(\mu, B_\infty^n) B_\infty^n) \leq \exp \brac{C \log(e+n) \brac{ \frac{p}{t^2} + \frac{p}{t} } } \;\;\; \forall p \geq 1 \;\; \forall t > 0 .
\]
\end{thm}
\begin{proof}
Since the expression on the left-hand-side is invariant under scaling of $\mu$, we may assume that  $m_1(\mu,B_\infty^n) = 1$. 

Given $p \geq 1$ and $t \geq 1/n$, we run The Program for $K = B_\infty^n$, with $\L_m = \set{B_\infty^m}$. Theorem \ref{thm:part1-cubes} verifies Part 1 of The Program with $D=C_1 \log(e+n)$, $R = C_2 \log \log(e+n)$, $A=1$, $B=1$ and $q_m=1$. Part 2 of The Program regarding Weak Generalized Regular Dual Sudakov, with parameters $q_m = 1$ and $\varphi_t(x) = C' \max(x^{1/3}/t^{2/3}, \sqrt{x} / \sqrt{t})$, is established in Theorem \ref{thm:part2-smalldiam} (recalling (\ref{eq:cover-pack})). Since $\varphi_t^{-1}(y) \simeq \min(y^{3} t^2, y^2 t)$, we have $\varphi_t^{-1}(1/(4A)) \simeq \min(t,t^2)$, and Theorem \ref{thm:full-program} yields the  asserted estimate in the range $t \geq 1/n$. The estimate remains valid after adjustment of constants (and in fact can be significantly improved) in the remaining uninteresting range $t \in (0,1/n)$ by Lemma \ref{lem:extend} and (\ref{eq:cover-pack}). This concludes the proof. 
\end{proof}

\begin{cor}[Generalized Regular Dual Sudakov For Polytopes with Few Facets and Logarithmic Terms] \label{cor:RegularSudakovPolytopes}
For any origin-symmetric log-concave measure $\mu$ on $\Real^n$, and any origin-symmetric polytope $K \subset \Real^n$ with $2N$ facets, we have:
\[
M(Z_p(\mu) , t C \log \log (e+N) m_1(\mu, K) K) \leq \exp \brac{C \log(e+N) \brac{ \frac{p}{t^2} + \frac{p}{t} } } \;\;\; \forall p \geq 1 \;\; \forall t > 0 .
\]
\end{cor}

\begin{proof}
Any $K \subset \Real^n$ as in the assertion is the unit-ball of an $n$-dimensional subspace $E$ (which we identify with $\Real^n$) of $\ell_\infty^N$ (in an appropriate basis), so that $K = B_\infty^N \cap E$. Let $\nu$ denote a compactly supported origin-symmetric log-concave probability measure on $E^\perp$, and let $\set{\nu_k}$ denote rescaled copies of $\nu$ which weakly converge to the delta-measure at the origin of $E^\perp$. Let $\mu$ be an origin-symmetric log-concave measure on $E$, which we may assume by approximation is compactly supported as well. Denote the product measure $\mu_k := \mu \otimes \nu_k$, which clearly has even log-concave density on $\Real^N$. By Theorem \ref{thm:RegularSudakovCubes} and Lemma \ref{lem:Guedon} applied to $\mu_k$ on $\Real^N$, we have for fixed $p \geq 1$ and $t > 0$:
\begin{equation} \label{eq:polytope-proof}
M(Z_p(\mu_k) , t C \log \log (e+N) I_1(\mu_k, B_\infty^N) B_\infty^N) \leq \exp \brac{C \log(e+N) \brac{ \frac{p}{t^2} + \frac{p}{t} } } .
\end{equation}
Note that by integrating against a bounded continuous function on $\Real^N$ and applying the Fubini and Lebesgue Dominant Convergence theorems, it follows that $\mu_k$ weakly converge to $\mu$. As all measures are uniformly compactly supported, it follows that $I_1(\mu_k,B_\infty^N)$ converges to $I_1(\mu,B_\infty^N) = I_1(\mu,K)$. 
In addition, since $Z_p(\mu) \times \set{0} \subset Z_p(\mu_k)$, 
it follows by definition and monotonicity of the packing numbers that for any $s > 0$:
\[
M(Z_p(\mu), s K) = M(Z_p(\mu) \times \set{0} , s B_\infty^N \cap E) = M(Z_p(\mu) \times \set{0} , s B_\infty^N) \leq M(Z_p(\mu_k) , s B_\infty^N) . 
\]
Combining the above observations, we obtain for large enough $k$:
\begin{align*}
& M(Z_p(\mu) ,  2 t C \log \log (e+N) I_1(\mu, K) K) \leq M(Z_p(\mu) ,  t C \log \log (e+N) I_1(\mu_k,B_\infty^N) K) \\
& \leq M(Z_p(\mu_k) , t C \log \log (e+N) I_1(\mu_k,B_\infty^N) B_\infty^N) .
\end{align*} 
Together with (\ref{eq:polytope-proof}) and another application of Lemma \ref{lem:Guedon}, the assertion follows after a possible readjustment of constants. 
\end{proof}

\section{Concluding Remarks} \label{sec:conclude}

\subsection{Generalized Regular Sudakov Minoration: Ellipsoids} \label{subsec:conclude-ell}

Recall that the following estimate was established in Theorem \ref{thm:RegularSudakovEllipsoids}:
\begin{equation} \label{eq:conclude-ellipsoids}
M(Z_p(\mu) , t m_1(\mu, \Eps) \Eps) \leq \exp \brac{C \brac{ \frac{p}{t^2} + \frac{p}{t} } } \;\;\; \forall p \geq 1 \;\; \forall t > 0 ,
\end{equation}
for any origin-symmetric log-concave measure $\mu$ on $\Real^n$ and any (origin-symmetric) ellipsoid $\Eps \subset \Real^n$. 
Let us expand on some of the comments regarding this estimate given in the Introduction. 

\medskip

In terms of sharpness, first recall that $Z_p(\mu) \subset p I_1(\mu,\Eps) \Eps$, and so the left-hand-side is $1$ for $t \geq C' p$ and the estimate is of the correct order (up to the value of $C > 0$) in that range. Moreover, our estimate yields the correct worst-case behavior in the range $t \in [1,C' p]$ as well. This is easily seen by degenerating $\mu$ to a one-dimensional (two-sided) exponential measure, in which case $Z_p(\mu)$ approximates an interval of length of order $p$, and $m_1(\mu,B_2^n)$ is of the order of $1$. 

\smallskip

We now claim that (\ref{eq:conclude-ellipsoids}) yields the correct worst-case behavior for all $t \in [\sqrt{p/n} , 1]$. To see this, set $\mu$ to be the standard Gaussian measure $\gamma_n$ so that $Z_p(\gamma_n) \simeq \sqrt{p} B_2^n$, and let $\Eps$ denote the cylinder $\sqrt{k} B_2^k \times \Real^{n-k}$ (which we think of as a degenerate ellipsoid, as it can obviously be approximated by proper ones). Clearly $m_1(\gamma_n,\Eps) \simeq I_2(\gamma_n,\Eps) = 1$, and we have by the volumetric estimate (\ref{eq:volumetric}):
\[
M(\sqrt{p} B_2^n , t_0 \Eps) = M(\sqrt{p} B_2^k , t_0 \sqrt{k} B_2^k) \geq e^k 
\]
for $t_0 = \frac{1}{2 e} \sqrt{p/k}$. Consequently, we confirm that for an appropriate constant $c >0$:
\[
M(Z_p(\gamma_n) , c \sqrt{p/k} \; m_1(\gamma_n, \Eps) \Eps)  \geq e^k ,
\]
and letting $k$ range from $\lceil p \rceil$ to $n$, the sharpness of (\ref{eq:conclude-ellipsoids}) for all $t \in [\sqrt{p/n} , 1]$ is established. 
When $t \in (0,\sqrt{p/n})$ the estimate is definitely loose, as simply seen by a volumetric argument (as explained in Lemma \ref{lem:extend}); however, we do not try to improve the estimate in this uninteresting range.

\medskip

As already mentioned in the Introduction, when $\mu$ is isotropic and $\Eps = B_2^n$ (and hence $m_1(\mu,B_2^n) \simeq \sqrt{n}$), the estimate (\ref{eq:conclude-ellipsoids}) was already obtained by Giannopoulos--Paouris--Valettas in \cite[Proposition 5.1]{GPV-ImprovedPsi2} (for $t \geq \sqrt{p/n}$, but the same estimate remains valid for all $t >0$ by Lemma \ref{lem:extend}) using a delicate refinement of Talagrand's approach for proving the (dual) Sudakov Minoration. It is possible to extend the latter approach to the non-isotropic setting, yielding the following estimate (we refrain from providing the details): 
\begin{equation} \label{eq:conclude-ellipsoids-paper1}
M(Z_p(\mu), t m_p(\mu, \Eps) \Eps) \leq \exp \brac{C \frac{p}{t^2}} \;\;\; \forall t \in [0,1] . 
\end{equation}
While (\ref{eq:conclude-ellipsoids}) improves upon (\ref{eq:conclude-ellipsoids-paper1}) in the range $t \geq 1$, note that (\ref{eq:conclude-ellipsoids-paper1}) involves the smaller quantile $m_p(\mu,\Eps) \leq m_1(\mu,\Eps)$, so these two estimates are ultimately incomparable. 
An alternative proof of (\ref{eq:conclude-ellipsoids}) in the isotropic case was obtained in \cite{GPV-DistributionOfPsi2} using a very similar approach to the one we employ in this work.
 For $p \geq \sqrt{n} \log^2(1+n)$, improved covering estimates in the range $t \in [\log^2(1+n) , p/\sqrt{n}]$  have been obtained for the isotropic case in \cite[Subsection 3.3]{EMilman-IsotropicMeanWidth}. In the non-isotropic case, a general formula in terms of the eigenvalues $\set{\lambda_i^2}_{i=1}^n$ of $\Cov(\mu)$ may also be obtained by employing Theorem \ref{thm:Sudakov} (improved Sudakov Minoration) and the estimate on $M^*(Z_p(\mu))$ from \cite[Theorem 1.3]{EMilman-IsotropicMeanWidth}; as $m_1(\mu,B_2^n) \simeq I_2(\mu,B_2^n) = \sum_{i=1}^n \lambda_i^2$, this results in possible improvements over (\ref{eq:conclude-ellipsoids}) in a certain range of the parameters $\set{\lambda_i} , p , t$; we leave the details to the interested reader. 

\subsection{Generalized Regular Sudakov Minoration: Cubes}

We now turn to the estimate established in Theorem \ref{thm:RegularSudakovCubes}:
\begin{equation} \label{eq:conclude-cubes}
M(Z_p(\mu) , t C \log \log (e+n) m_1(\mu, B_\infty^n) B_\infty^n) \leq \exp \brac{C \log(e+n) \brac{ \frac{p}{t^2} + \frac{p}{t} } } \;\;\; \forall p \geq 1 \;\; \forall t > 0 
\end{equation}
for any origin-symmetric log-concave measure $\mu$ on $\Real^n$. Up to the logarithmic terms above, this estimate is again seen to be sharp in the range $t \geq 1$, exactly as in the preceding analysis for ellipsoids. 

\smallskip

In the range $t \in [\sqrt{p/n} , 1]$, the estimate (\ref{eq:conclude-cubes}) remains sharp up to logarithmic terms in the dimension. To see this, set again $\mu$ to be the Gaussian measure $\gamma_n$, for which it is well-known that $m_1(\gamma_n , B_\infty^n) \simeq \sqrt{\log(1+n)}$. Applying the precise covering estimates of Sch\"utt \cite[Theorem 1]{Schutt-Entropy-numbers}, we have:
\[
M\brac{Z_p(\gamma_n) , C \frac{\sqrt{\log (1+n/k)}}{\sqrt{\log (1+n)}} \frac{\sqrt{p}}{\sqrt{k}} m_1(\gamma_n , B_\infty^n)  B_\infty^n} 
\geq M\brac{ B_2^n , C' \frac{\sqrt{\log (1+n/k)}}{\sqrt{k}}  B_\infty^n} 
 \geq e^k ,
\]
for all $\log(1+n) \leq k \leq n$. This confirms the sharpness of (\ref{eq:conclude-cubes}) up to the logarithmic terms there for all $t \in [\sqrt{p/n^\alpha},1]$ for any fixed $\alpha \in (0,1)$, and up to an additional $\log(e+n)$ term in the range $t \in [\sqrt{p/n},\sqrt{p/n^\alpha}]$. Curiously, in the latter range, this additional term yields a packing estimate for $M(Z_p(\gamma_n) , t m_1(\gamma_n, B_\infty^n) B_\infty^n)$ which is even better than the expected $\exp(C \frac{p}{t^2})$, and we do not know 
whether this is indeed the worst-possible expected behaviour for a general $\mu$. As in the case of ellipsoids, the estimate is definitely loose in the range $t \in (0,\sqrt{p/n})$ by a simple volumetric estimate. 

\subsection{Completing The Program}

The results we obtain in this work completely resolve Part 3 of The Program, and almost entirely Part 2 as well. For instance, if the initial log-concave probability measure $\mu$ is assumed $1$-pure (e.g. super-Gaussian, sub-Gaussian or unconditional), then by Proposition \ref{prop:pure-1}, so will be all of its marginals $\nu \in \M_m$, for which we have a Weak Sudakov Minoration result by Theorem \ref{thm:part2-pure}. In particular, up to the Slicing Problem, Part 2 is completely established. 

\smallskip

Consequently, it is clear that the main remaining challenge in completing The Program lies in establishing Part 1 of The Program. This is a significant challenge even for some specific convex bodies $K$ besides ellipsoids, such as for $K = B_1^n$. To carry out this Separation Dimension-Reduction step, it seems that we would need to employ other measures on the Grassmannian $G_{n,m}$ besides the uniform Haar measure, upon which most of the (Euclidean) Asymptotic Geometric Analysis theory is built. In our opinion, this is a fascinating challenge, which we plan to explore in a future work.

\setlength{\bibspacing}{2pt}

\bibliographystyle{plain}
\bibliography{../../../../ConvexBib}

\def\cprime{$'$} \def\textasciitilde{$\sim$}
\begin{thebibliography}{10}

\bibitem{AMST-Duality-For-K-Convex}
S.~Artstein, V.~Milman, S.~Szarek, and N.~Tomczak-Jaegermann.
\newblock On convexified packing and entropy duality.
\newblock {\em Geom. Funct. Anal.}, 14(5):1134--1141, 2004.

\bibitem{AMS-Duality-For-Ball}
S.~Artstein, V.~Milman, and S.~J. Szarek.
\newblock Duality of metric entropy.
\newblock {\em Ann. of Math. (2)}, 159(3):1313--1328, 2004.

\bibitem{AGA-Book-I}
S.~Artstein-Avidan, A.~Giannopoulos, and V.~D. Milman.
\newblock {\em Asymptotic geometric analysis. {P}art {I}}, volume 202 of {\em
  Mathematical Surveys and Monographs}.
\newblock American Mathematical Society, Providence, RI, 2015.

\bibitem{Ball-PhD}
K.~Ball.
\newblock {\em Isometric problems in $l_p$ and sections of convex sets}.
\newblock PhD thesis, Cambridge, 1986.

\bibitem{Ball-kdim-sections}
K.~Ball.
\newblock Logarithmically concave functions and sections of convex sets in
  $\mathbb{R}^n$.
\newblock {\em Studia Math.}, 88(1):69--84, 1988.

\bibitem{BarlowMarshallProschan}
R.~E. Barlow, A.~W. Marshall, and F.~Proschan.
\newblock Properties of probability distributions with monotone hazard rate.
\newblock {\em Ann. Math. Statist.}, 34:375--389, 1963.

\bibitem{BGMN}
F.~Barthe, O.~Gu{\'e}don, S.~Mendelson, and A.~Naor.
\newblock A probabilistic approach to the geometry of the {$l\sp n\sb p$}-ball.
\newblock {\em Ann. Probab.}, 33(2):480--513, 2005.

\bibitem{BerwaldMomentComparison}
L.~Berwald.
\newblock Verallgemeinerung eines {M}ittelwertsatzes von {J}. {F}avard f\"ur
  positive konkave {F}unktionen.
\newblock {\em Acta Math.}, 79:17--37, 1947.

\bibitem{Borell-logconcave}
Ch. Borell.
\newblock Convex measures on locally convex spaces.
\newblock {\em Ark. Mat.}, 12:239--252, 1974.

\bibitem{BourgainMaximalFunctionsOnConvexBodies}
J.~Bourgain.
\newblock On high-dimensional maximal functions associated to convex bodies.
\newblock {\em Amer. J. Math.}, 108(6):1467--1476, 1986.

\bibitem{Bourgain-LK}
J.~Bourgain.
\newblock On the distribution of polynomials on high dimensional convex sets.
\newblock In {\em Geometric Aspects of Functional Analysis}, volume 1469 of
  {\em Lecture Notes in Math.}, pages 127--137. Springer-Verlag, 1991.

\bibitem{Bourgain-Psi-2-Bodies}
J.~Bourgain.
\newblock On the isotropy-constant problem for "psi-2"-bodies.
\newblock In {\em Geometric Aspects of Functional Analysis}, volume 1807 of
  {\em Lecture Notes in Mathematics}, pages 114--121. Springer, 2001-2002.

\bibitem{GreekBook}
S.~Brazitikos, Giannopoulos A., Valettas P., and Vritsiou B.-H.
\newblock {\em Geometry of Isotropic Convex Bodies}, volume 196 of {\em
  Mathematical Surveys and Monographs}.
\newblock Amer. Math. Soc., 2014.

\bibitem{DafnisGiannopoulosTsolomitis-RandomPolytopes}
N.~Dafnis, A.~Giannopoulos, and A.~Tsolomitis.
\newblock Asymptotic shape of a random polytope in a convex body.
\newblock {\em J. Funct. Anal.}, 257(9):2820--2839, 2009.

\bibitem{Davis-etal-Lemma}
W.~J. Davis, V.~D. Milman, and N.~Tomczak-Jaegermann.
\newblock The distance between certain $n$-dimensional banach spaces.
\newblock {\em Israel Journal of Mathematics}, 39:1--15, 1981.

\bibitem{GardnerSurveyInBAMS}
R.~J. Gardner.
\newblock The {B}runn-{M}inkowski inequality.
\newblock {\em Bull. Amer. Math. Soc. (N.S.)}, 39(3):355--405, 2002.

\bibitem{GiannopoulosEMilman-IsotropicM}
A.~Giannopoulos and E.~Milman.
\newblock {$M$}-estimates for isotropic convex bodies and their
  {$L_q$}-centroid bodies.
\newblock In {\em Geometric aspects of functional analysis}, volume 2116 of
  {\em Lecture Notes in Math.}, pages 159--182. Springer, Cham, 2014.

\bibitem{GiannopoulosPajorPaourisPsi2}
A.~Giannopoulos, A.~Pajor, and G.~Paouris.
\newblock A note on subgaussian estimates for linear functionals on convex
  bodies.
\newblock {\em Proc. Amer. Math. Soc.}, 135(8):2599--2606 (electronic), 2007.

\bibitem{GPV-ImprovedPsi2}
A.~Giannopoulos, G.~Paouris, and P.~Valettas.
\newblock On the existence of subgaussian directions for log-concave measures.
\newblock In {\em Concentration, functional inequalities and isoperimetry},
  volume 545 of {\em Contemp. Math.}, pages 103--122. Amer. Math. Soc.,
  Providence, RI, 2011.

\bibitem{GPV-DistributionOfPsi2}
A.~Giannopoulos, G.~Paouris, and P.~Valettas.
\newblock On the distribution of the {$\psi_2$}-norm of linear functionals on
  isotropic convex bodies.
\newblock In {\em Geometric aspects of functional analysis}, volume 2050 of
  {\em Lecture Notes in Math.}, pages 227--253. Springer, Heidelberg, 2012.

\bibitem{GiannopoulosMilmanHandbook}
A.~A. Giannopoulos and V.~D. Milman.
\newblock Euclidean structure in finite dimensional normed spaces.
\newblock In {\em Handbook of the geometry of Banach spaces, Vol. I}, pages
  707--779. North-Holland, Amsterdam, 2001.

\bibitem{Guedon-extension-to-negative-p}
O.~Gu{\'{e}}don.
\newblock Kahane-khinchine type inequalities for negative exponent.
\newblock {\em Mathematika}, 46:165--173, 1999.

\bibitem{GuedonEMilmanInterpolating}
O.~Gu{\'{e}}don and E.~Milman.
\newblock Interpolating thin-shell and sharp large-deviation estimates for
  isotropic log-concave measures.
\newblock {\em Geom. Func. Anal.}, 21(5):1043--1068, 2011.

\bibitem{Hartzoulaki-PhD}
M.~Hartzoulaki.
\newblock {\em Probabilistic methods in the theory of convex bodies}.
\newblock PhD thesis, University of Crete, March 2003.

\bibitem{JohnsonLindenstraussLemma}
W.~B. Johnson and J.~Lindenstrauss.
\newblock Extensions of {L}ipschitz mappings into a {H}ilbert space.
\newblock In {\em Conference in modern analysis and probability ({N}ew {H}aven,
  {C}onn., 1982)}, volume~26 of {\em Contemp. Math.}, pages 189--206. Amer.
  Math. Soc., Providence, RI, 1984.

\bibitem{JohnsonNaor-JL}
W.~B. Johnson and A.~Naor.
\newblock The {J}ohnson-{L}indenstrauss lemma almost characterizes {H}ilbert
  space, but not quite.
\newblock {\em Discrete Comput. Geom.}, 43(3):542--553, 2010.

\bibitem{KlartagPerturbationsWithBoundedLK}
B.~Klartag.
\newblock On convex perturbations with a bounded isotropic constant.
\newblock {\em Geom. and Funct. Anal.}, 16(6):1274--1290, 2006.

\bibitem{KlartagCLPpolynomial}
B.~Klartag.
\newblock Power-law estimates for the central limit theorem for convex sets.
\newblock {\em J. Funct. Anal.}, 245:284--310, 2007.

\bibitem{Klartag-Psi2}
B.~Klartag.
\newblock Uniform almost sub-{G}aussian estimates for linear functionals on
  convex sets.
\newblock {\em Algebra i Analiz}, 19(1):109--148, 2007.

\bibitem{KlartagEMilman-2-Convex}
B.~Klartag and E.~Milman.
\newblock On volume distribution in $2$-convex bodies.
\newblock {\em Israel J. Math.}, 164:221--249, 2008.

\bibitem{KlartagEMilmanLowerBoundsOnZp}
B.~Klartag and E.~Milman.
\newblock Centroid bodies and the logarithmic {L}aplace {T}ransform - a unified
  approach.
\newblock {\em J. Func. Anal.}, 262(1):10--34, 2012.

\bibitem{KlartagMilmanLogConcave}
B.~Klartag and V.~D. Milman.
\newblock Geometry of log-concave functions and measures.
\newblock {\em Geom. Dedicata}, 112:169--182, 2005.

\bibitem{Klartag-Vershynin}
B.~Klartag and R.~Vershynin.
\newblock Small ball probability and {D}voretzky's theorem.
\newblock {\em Israel J. Math.}, 157:193--207, 2007.

\bibitem{Latala-SudakovMinorationAndGenericChaining}
R.~Lata{\l}a.
\newblock Sudakov minoration principle and supremum of some processes.
\newblock {\em Geom. Funct. Anal.}, 7(5):936--953, 1997.

\bibitem{Latala-GeneralizedSudakovMinoration}
R.~Lata{\l}a.
\newblock Sudakov-type minoration for log-concave vectors.
\newblock {\em Studia Math.}, 223(3):251--274, 2014.

\bibitem{LatalaTkocz-SudakovMinorationForRegularProduct}
R.~Lata{\l}a and T.~Tkocz.
\newblock A note on suprema of canonical processes based on random variables
  with regular moments.
\newblock {\em Electron. J. Probab.}, 20:no. 36, 17, 2015.

\bibitem{LatalaJacobInfConvolution}
R.~Lata{\l}a and J.~O. Wojtaszczyk.
\newblock On the infimum convolution inequality.
\newblock {\em Studia Math.}, 189(2):147--187, 2008.

\bibitem{LedouxTalagrand-Book}
M.~Ledoux and M.~Talagrand.
\newblock {\em Probability in {B}anach spaces}, volume~23 of {\em Ergebnisse
  der Mathematik und ihrer Grenzgebiete (3) [Results in Mathematics and Related
  Areas (3)]}.
\newblock Springer-Verlag, Berlin, 1991.
\newblock Isoperimetry and processes.

\bibitem{LitvakMilmanPajor-QuasiConvex}
A.~E. Litvak, V.~D. Milman, and A.~Pajor.
\newblock The covering numbers and ``low {$M^\ast$}-estimate'' for quasi-convex
  bodies.
\newblock {\em Proc. Amer. Math. Soc.}, 127(5):1499--1507, 1999.

\bibitem{LutwakZhang-IntroduceLqCentroidBodies}
E.~Lutwak and G.~Zhang.
\newblock Blaschke-{S}antal\'o inequalities.
\newblock {\em J. Differential Geom.}, 47(1):1--16, 1997.

\bibitem{MendelsonVershynin-CombDim}
S.~Mendelson and R.~Vershynin.
\newblock Entropy and the combinatorial dimension.
\newblock {\em Invent. Math.}, 152(1):37--55, 2003.

\bibitem{EMilman-DualMixedVolumes}
E.~Milman.
\newblock Dual mixed volumes and the slicing problem.
\newblock {\em Adv. Math.}, 207(2):566--598, 2006.

\bibitem{EMilman-Duality-of-Entropy}
E.~Milman.
\newblock A remark on two duality relations.
\newblock {\em Integral Equations and Operator Theory}, 57(2):217--228, 2007.

\bibitem{EMilman-IsotropicMeanWidth}
E.~Milman.
\newblock On the mean-width of isotropic convex bodies and their associated
  {$L_p$}-centroid bodies.
\newblock {\em Int. Math. Res. Not.}, (11):3408--3423, 2015.

\bibitem{Milman-Pajor-LK}
V.~D. Milman and A.~Pajor.
\newblock Isotropic position and interia ellipsoids and zonoids of the unit
  ball of a normed $n$-dimensional space.
\newblock In {\em Geometric Aspects of Functional Analysis}, volume 1376 of
  {\em Lecture Notes in Mathematics}, pages 64--104. Springer-Verlag,
  1987-1988.

\bibitem{Milman-Schechtman-Book}
V.~D. Milman and G.~Schechtman.
\newblock {\em Asymptotic theory of finite-dimensional normed spaces}, volume
  1200 of {\em Lecture Notes in Mathematics}.
\newblock Springer-Verlag, Berlin, 1986.
\newblock With an appendix by M. Gromov.

\bibitem{MilmanSzarek-GeometricLemma}
V.~D. Milman and S.~J. Szarek.
\newblock A geometric lemma and duality of entropy numbers.
\newblock In {\em Geometric aspects of functional analysis}, volume 1745 of
  {\em Lecture Notes in Math.}, pages 191--222. Springer, Berlin, 2000.

\bibitem{PajorTomczakLowMStar}
A.~Pajor and N.~Tomczak-Jaegermann.
\newblock Subspaces of small codimension of finite-dimensional {B}anach spaces.
\newblock {\em Proc. Amer. Math. Soc.}, 97(4):637--642, 1986.

\bibitem{Paouris-Small-Diameter}
G.~Paouris.
\newblock $\psi_2$-estimates for linear functionals on zonoids.
\newblock In {\em Geometric Aspects of Functional Analysis}, volume 1807 of
  {\em Lecture Notes in Mathematics}, pages 211--222. Springer, 2001-2002.

\bibitem{PaourisPsi2Behaviour}
G.~Paouris.
\newblock On the {$\psi_2$}-behaviour of linear functionals on isotropic convex
  bodies.
\newblock {\em Studia Math.}, 168(3):285--299, 2005.

\bibitem{Paouris-IsotropicTail}
G.~Paouris.
\newblock Concentration of mass on convex bodies.
\newblock {\em Geom. Funct. Anal.}, 16(5):1021--1049, 2006.

\bibitem{Paouris-MarginalsOfProducts}
G.~Paouris.
\newblock On the isotropic constant of marginals.
\newblock {\em Studia Math.}, 212(3):219--236, 2012.

\bibitem{PaourisSmallBall}
G.~Paouris.
\newblock Small ball probability estimates for log-concave measures.
\newblock {\em Trans. Amer. Math. Soc.}, 364(1):287--308, 2012.

\bibitem{Pietsch-Book}
A.~Pietsch.
\newblock {\em Theorie der {O}peratorenideale ({Z}usammenfassung)}.
\newblock Friedrich-Schiller-Universit\"at, Jena, 1972.

\bibitem{Pisier-Book}
G.~Pisier.
\newblock {\em The volume of convex bodies and {B}anach space geometry},
  volume~94 of {\em Cambridge Tracts in Mathematics}.
\newblock Cambridge University Press, Cambridge, 1989.

\bibitem{RudelsonVershynin-CombDim}
M.~Rudelson and R.~Vershynin.
\newblock Combinatorics of random processes and sections of convex bodies.
\newblock {\em Ann. of Math. (2)}, 164(2):603--648, 2006.

\bibitem{Schneider-Book}
R.~Schneider.
\newblock {\em Convex bodies: the {B}runn-{M}inkowski theory}, volume~44 of
  {\em Encyclopedia of Mathematics and its Applications}.
\newblock Cambridge University Press, Cambridge, 1993.

\bibitem{Schutt-Entropy-numbers}
C.~Sch{\"u}tt.
\newblock Entropy numbers of diagonal operators between symmetric {B}anach
  spaces.
\newblock {\em J. Approx. Theory}, 40(2):121--128, 1984.

\bibitem{SudakovMinoration}
V.~N. Sudakov.
\newblock A remark on the criterion of continuity of {G}aussian sample
  function.
\newblock In {\em Proceedings of the {S}econd {J}apan-{USSR} {S}ymposium on
  {P}robability {T}heory ({K}yoto, 1972)}, pages 444--454. Lecture Notes in
  Math., Vol. 330. Springer, Berlin, 1973.

\bibitem{Talagrand-RegularityOfGaussianProcesses}
M.~Talagrand.
\newblock Regularity of {G}aussian processes.
\newblock {\em Acta Math.}, 159(1-2):99--149, 1987.

\bibitem{Talagrand-GenericChaining}
M.~Talagrand.
\newblock Majorizing measures: the generic chaining.
\newblock {\em Ann. Probab.}, 24(3):1049--1103, 1996.

\bibitem{Talagrand-Book}
M.~Talagrand.
\newblock {\em The generic chaining}.
\newblock Springer-Verlag, Berlin, 2005.
\newblock Upper and lower bounds of stochastic processes.

\bibitem{Vritsiou-ExtendingKM}
B.-H. Vritsiou.
\newblock Further unifying two approaches to the hyperplane conjecture.
\newblock {\em Int. Math. Res. Not. IMRN}, (6):1493--1514, 2014.

\end{thebibliography}

 \end{document}